\documentclass{amsart}
\usepackage{amssymb,amsfonts, color,epsf}
\usepackage [cmtip,arrow]{xy}
\xyoption{all}
\usepackage {pb-diagram,pb-xy}
\usepackage{enumerate}
\usepackage{accents}
\usepackage[noautoscale]{youngtab}
\usepackage{subfigure}
\usepackage[utf8]{inputenc}

\ifx\pdfpageheight\undefined
\usepackage[dvips,colorlinks=true,linkcolor=blue,citecolor=red,%
urlcolor=green]{hyperref}
\usepackage[dvips]{graphicx}
\makeatletter
\edef\Gin@extensions{\Gin@extensions,.mps}
\makeatother
\else
\usepackage[pdftex]{graphicx}
\usepackage[bookmarksopen=false,pdftex=true,breaklinks=true,%
backref=page,pagebackref=true,plainpages=false,%
hyperindex=true,pdfstartview=FitH,colorlinks=true,%
pdfpagelabels=true,colorlinks=true,linkcolor=blue,%
citecolor=red,urlcolor=green,hypertexnames=false%
]%
{hyperref}
\fi
\usepackage{bm}

\usepackage{tikz-cd}
\makeatletter
\tikzset{
	column sep/.code=\def\pgfmatrixcolumnsep{\pgf@matrix@xscale*(#1)},
	row sep/.code   =\def\pgfmatrixrowsep{\pgf@matrix@yscale*(#1)},
	matrix xscale/.code=%
	\pgfmathsetmacro\pgf@matrix@xscale{\pgf@matrix@xscale*(#1)},
	matrix yscale/.code=%
	\pgfmathsetmacro\pgf@matrix@yscale{\pgf@matrix@yscale*(#1)},
	matrix scale/.style={/tikz/matrix xscale={#1},/tikz/matrix yscale={#1}}}
\def\pgf@matrix@xscale{1}
\def\pgf@matrix@yscale{1}
\makeatother

\usepackage{amsmath,amssymb,amsfonts}
\newtheorem{theorem}{Theorem}
\newtheorem{lemma}{Lemma}[section]

\newtheorem{proposition}{Proposition}[section]

\newtheorem*{claim*}{Claim}

\newtheorem{problem}{Problem}%
\newtheorem{assumption}{Assumption}%

\newtheorem*{theorem*}{Theorem}
\newtheorem*{corollary*}{Corollary}
\theoremstyle{definition}
\newtheorem{definition}{Definition}[section]
\newtheorem{example}{Example}

\newtheorem{notation}{Notation}

\usepackage{algorithm}% http://ctan.org/pkg/algorithms
\usepackage{algpseudocode}% http://ctan.org/pkg/algorithmicx

\usepackage{tikz}

\algnewcommand\algorithmicinput{\textbf{Input:}}
\algnewcommand\INPUT{\item[\algorithmicinput]}
\algnewcommand\algorithmicoutput{\textbf{Output:}}
\algnewcommand\OUTPUT{\item[\algorithmicoutput]}

\algnewcommand\algorithmicproc{\textbf{Procedure:}}
\algnewcommand\PROCEDURE{\item[\algorithmicproc]}
\algnewcommand\algorithmiccomplexity{\textbf{Complexity:}}
\algnewcommand\COMPLEXITY{\item[\algorithmiccomplexity]}

\newlength{\continueindent}
\usepackage{etoolbox}
\makeatletter
\newcommand*{\ALG@customparshape}{\parshape 2 \leftmargin \linewidth \dimexpr\ALG@tlm+\continueindent\relax \dimexpr\linewidth+\leftmargin-\ALG@tlm-\continueindent\relax}
\apptocmd{\ALG@beginblock}{\ALG@customparshape}{}{\errmessage{failed to patch}}
\makeatother
\theoremstyle{remark}
\newtheorem{remark}{Remark}

\theoremstyle{observation}

\definecolor{DarkBlue}{rgb}{0,0.1,0.55}

\numberwithin{equation}{section}

\newcommand {\hide}[1]{}
\newcommand {\sign} {\mbox{\bf sign}}
\newcommand{\zero}{\mbox{\bf zero}}

\newcommand {\junk}[1]{}

\newcommand {\R} {\mathrm{R}}

\newcommand {\SI}         {\mbox{\rm SIGN}}

\newcommand {\RR} {{\mathcal R}}

\newcommand {\Der} {{\rm Der}}

\newcommand {\T}      {\mbox{\rm T}}

\newcommand{\card}{\mathrm{card}}

\def\addots{\mathinner{\mkern1mu
		\raise1pt\vbox{\kern7pt\hbox{.}}
		\mkern2mu\raise4pt\hbox{.}\mkern2mu
		\raise7pt\hbox{.}\mkern1mu}}

\DeclareMathOperator{\trace}{Tr}

\DeclareMathOperator{\vectorize}     {vec}
\DeclareMathOperator{\column} {Col}
\DeclareMathOperator{\tdeg} {tdeg}
\DeclareMathOperator{\real} {Re}
\DeclareMathOperator{\imag} {Im}
\DeclareMathOperator{\content} {Cont}

\newcommand{\x}{\mathbf{x}}

\newcommand{\y}{\mathbf{y}}

	\title[On the complexity of analyticity in SDO]
	{
		On the complexity of analyticity in semi-definite optimization
	}
        \author{Saugata Basu}
	\address{Department of Mathematics, 
		Purdue University, West Lafayette, IN 47905, U.S.A.}
	\email{sbasu@purdue.edu}
	
	\author{Ali Mohammad-Nezhad}
	\address{Department of Mathematical Sciences, 
		Carnegie Mellon University, Pittsburgh, PA 15123, U.S.A.}
	\email{anezhad@andrew.cmu.edu}

\begin{document}
\begin{abstract}
It is well-known that the central path of semi-definite optimization, unlike linear optimization, has no analytic extension to $\mu = 0$ in the absence of the strict complementarity condition. In this paper, we show the existence of a positive integer $\rho$ by which the reparametrization $\mu \mapsto \mu^{\rho}$ recovers the analyticity of the central path at $\mu = 0$. We investigate the complexity of computing $\rho$ using algorithmic real algebraic geometry and the theory of complex algebraic curves. We prove that the optimal $\rho$ is bounded by $2^{O(m^2+n^2m+n^4)}$, where $n$ is the matrix size and $m$ is the number of affine constraints. Our approach leads to a symbolic algorithm, based on the Newton-Puiseux algorithm, which computes a feasible $\rho$ using $2^{O(m+n^2)}$ arithmetic operations.
\end{abstract}


\maketitle
\tableofcontents

%%%%%%%%
%New Section
%%%%%%%%
\section{Introduction}\label{intro}
\subsection{Semi-definite optimization} We denote by $\mathbb{S}^n$ the inner product space of 
$n \times n$ symmetric matrices 
%%sb adds
with entries in $\mathbb{R}$
with the inner product $\langle C, X \rangle =\trace(CX)$. A pair of primal-dual \textit{semi-definite optimization} (SDO) problems is defined as 
 \begin{align*}
&(\mathrm{P}) \qquad v_{p}^*:=\inf_{X \in \mathbb{S}^n} \big \{\langle C,  X \rangle \mid \langle A^{i} , X \rangle=b_i, \quad  i=1,\ldots, m, \ X \succeq 0 \big \},\\
&(\mathrm{D}) \qquad v^*_d:=\sup_{(y,S) \in \mathbb{R}^m \times \mathbb{S}^n} \Big \{b^Ty  \mid  \sum_{i=1}^m y_i A^i+S=C, \ S \succeq 0, \ y \in \mathbb{R}^m \Big \},
 \end{align*}
where $A^i \in \mathbb{S}^n$ for $i=1,\ldots, m$, $C \in \mathbb{S}^n$, $b \in \mathbb{R}^m$, and $\succeq 0$ means positive semi-definite. A primal-dual vector $(X,y,S)$ is called a \textit{solution} if $(X,y,S) \in \mathrm{Sol(P)} \times \mathrm{Sol(D)}$, where 
\begin{align*}
&\mathrm{Sol(P)}:=\Big\{X \mid  \langle A^{i}, X \rangle=b_i, \quad  i=1,\ldots, m, \ \ X \succeq 0, \ \ \langle C,  X \rangle = v^*_p  \Big\}, \\
&\mathrm{Sol(D)}:=\bigg\{(y,S) \mid \sum_{i=1}^m y_i A^i+S=C, \ \ S \succeq 0,  \ b^Ty=v^*_d \bigg\}.
\end{align*}
%%%%%%%%
%New Notation
%%%%%%%%
\begin{notation}\label{matrix_vector_form}
We adopt the notation $(\cdot,\cdot,\cdots,\cdot)$ to represent vectors or side by side arrangement of matrices. We sometimes identify a symmetric matrix $X=(X_{ij})_{n \times n}$ with a column vector $x$ by stacking the columns of the matrix on the top of each other, i.e., we use the vector isomorphism 
\begin{equation}\label{matrix_identification}
\begin{aligned}
\vectorize&: \mathbb{S}^{n} \to \mathbb{R}^{n^2}\\
X &\mapsto \big(X_{11},\ldots, X_{1n}, X_{21}, \ldots,X_{2n},\ldots,X_{n1},\ldots, X_{nn}\big)^T.
\end{aligned}
\end{equation}
Following this notation, we also define $\mathbf{A}\!:=\big(\vectorize(A^1),\ldots,\vectorize(A^m)\big)^T$. 
\end{notation}

\subsection{Central path}
SDO problems can be solved ``efficiently" using path-following interior point methods (IPMs)~\cite{NN94}, where the \textit{central path} plays a prominent role. The central path is an analytic semi-algebraic function $\xi(\mu)\!:\!(0,\infty) \to \mathbb{S}^{n} \times \mathbb{R}^m \times \mathbb{S}^n$ whose graph $(\mu,\xi(\mu))$ satisfies  
\begin{align}\label{central_path}
\big\{\mathbf{A} \vectorize(X)&=b, \ \mathbf{A}^T y + \vectorize(S-C)=0, \ \vectorize(XS-\mu I_n)=0, \ X,S \succ 0\big\},
\end{align}
where $I_n$ is the identity matrix of size $n$, and $\succ 0$ means positive definite. For every fixed positive $\mu$, following Notation~\ref{matrix_vector_form}, there exists a unique $(X(\mu),y(\mu),S(\mu))$, so-called a \textit{central solution}, satisfying~\eqref{central_path}~\cite[Theorem~3.1]{Kl02}. The analyticity of the central path is immediate from the application of the analytic implicit function theorem to the nonsingular Jacobian at a central solution~\cite[Theorem~3.3]{Kl02}. Due to its semi-algebraicity~\cite[Proposition~3.18]{BPR06} and the boundedness of $\xi\mid_{(0,1]}$~\cite[Lemma~3.2]{Kl02}, 
 the central path converges%
\footnote{The first proof of convergence was given in~\cite[Theorem~A.3]{HKR02} based on the curve selection lemma~\cite[Lemma~3.1]{M68}.} as $\mu \downarrow 0$~\cite[Theorem~A.3]{HKR02} to a solution $(X^{**},y^{**},S^{**})$ in the relative interior of the solution set~\cite[Lemma~4.2]{GS98}.
%%%%%%%%
%New Definition
%%%%%%%%
\begin{definition}\label{maximal_strict_comp}
A solution $(X,y,S)$ is called \textit{strictly complementary} if $X+S \succ 0$. The \textit{strict complementarity condition} is said to hold if there exists a strictly complementary solution for $(\mathrm{P})-(\mathrm{D})$.
\end{definition} 

\vspace{5px}
\noindent
The following assumption is made throughout to guarantee the existence of the central path. This also guarantees that $\mathrm{Sol(P)} \times \mathrm{Sol(D)}$ is nonempty and compact.
\begin{assumption}\label{IND_IPC}
The matrices $A^i$ for $i=1,\ldots,m$ are linearly independent, and there exists a feasible $(X,y,S)$ such that $X,S \succ 0$.
\end{assumption}

\vspace{5px}
\noindent
The central path in case of SDO is analytic at $\mu = 0$ if and only if the limit point of the central path is strictly complementary. Unfortunately, the failure of analyticity impairs the convergence rate of path-following IPMs, see e.g.,~\cite{WW10}. This problem has been extensively studied for linear optimization, linear complementarity problems, and also SDO in the presence of the strict complementarity condition. We will survey some of these results in Section~\ref{survey}. Our paper addresses the case for SDO regardless of the strict complementarity condition.  

%%%%%%%%
%New Notation
%%%%%%%%
\begin{notation}\label{vector_form}
Using the identification~\eqref{matrix_identification}, the limit point $(X^{**},y^{**},S^{**})$ and a central solution $(X(\mu),y(\mu),S(\mu))$ are identified by vectors $v^{**}:=(x^{**};y^{**};s^{**})$ and $v(\mu):=(x(\mu);y(\mu);s(\mu))$, respectively. Furthermore, the coordinates of the central path are denoted by $v_i(\mu)$ for $i=1,\ldots,\bar{n}$, where 
\begin{align*}
\bar{n}:=m+2n^2.
\end{align*}
\end{notation}
%%%%%%%%
%New Notation
%%%%%%%%
\begin{notation}\label{Intermediates_Notation}
The indeterminates of the polynomials defining the semi-algebraic set~\eqref{central_path} are denoted by $V_1,\ldots, V_{\bar{n}}$.
\end{notation}

\section{Main results}
In this paper, we explore the issue of analyticity in the absence of the strict complementarity condition. Our main motivation behind studying the analyticity of the central path is the following key question, as originally stated in~\cite[Page~519]{PS04}. This is the analog of the problem on the central path in the case of linear complementarity problems~\cite{SW98,SW99}.

%%%%%%%%
%New Problem
%%%%%%%%
\begin{problem}\label{CP_Analytic_Extension}
Does there exist an integer $\rho > 0$ such that $\xi(\mu^{\rho})$ is analytic at $\mu = 0$?
\end{problem}
Here, we provide an affirmative answer to Problem~\ref{CP_Analytic_Extension} and investigate the complexity of computing a feasible $\rho$ by means of a symbolic procedure.
\subsection{Upper bound on the ramification index}
 For the purpose of complexity analysis, we assume the integrality of the data in $(\mathrm{P})-(\mathrm{D})$.
\begin{assumption}\label{integral_data}
The entries in $A^i$ for $i=1,\ldots,m$, $b$, and $C$ are all integers.
\end{assumption}
%%%%%%%%
%New Notation
%%%%%%%%
\begin{notation}\label{Input_bitsizes}
A bound on the bitsizes of the entries of $A^i$, $C$, $b$ is denoted by $\tau$.
\end{notation}
\noindent
In Section~\ref{Puiseux_Expansion}, we show that $\rho$ attains its optimal value at the least common multiple of the ramification indices of the Puiseux expansions of $v_i(\mu)$, which we denote by $q$, see Notation~\ref{vector_form}. As a consequence, we show that $\rho \in \mathbb{Z}_{+} q$, i.e., any positive integer multiple ($\ge 1$) of the ramification index would be also a feasible solution to Problem~\ref{CP_Analytic_Extension}. In particular, we prove an upper bound on the optimal $\rho$.
%%%%%%%%
%New Theorem
%%%%%%%%
\begin{theorem}\label{thm_reparametrization}
The optimal $\rho$ is bounded by $2^{O(m^2+n^2m+n^4)}$.
\end{theorem}

\begin{remark}
If we also take into account $m = O(n^2)$ from Assumption~\ref{IND_IPC} (because $A_1,\ldots,A_m$ are assumed to be linearly independent), then the optimal $\rho$ is $2^{O(n^4)}$. 
\end{remark}

\begin{remark}[Designing higher-order IPMs]
Our analytic reparametrization has the advantage that it is independent of the strict complementarity condition. In the words of the author in~\cite{H02}, this approach might be helpful in designing higher-order IPMs for SDO, and second-order conic optimization, with better local convergence than the regular IPMs.
\end{remark}
\noindent

\subsection{A symbolic procedure for computing $\rho$}
Our next contribution is a symbolic procedure, see Algorithm~\ref{analytic_param}, for an efficient computation of a feasible $\rho$ in Problem~\ref{CP_Analytic_Extension}. Using algorithmic real algebraic geometry and the theory of complex algebraic curves, we propose a symbolic algorithm, based on the Newton-Puiseux algorithm, which computes a feasible $\rho$ using $2^{O(m+n^2)}$ arithmetic operations, see Algorithm~\ref{analytic_param}. In the sequel, we prove that $\rho$ from Algorithm~\ref{analytic_param} is bounded by $2^{2^{O(m+n^2)}}$. The following theorem summarizes one of the main results of this paper.

%%%%%%%%
%New Theorem
%%%%%%%%
\begin{theorem}\label{Algorithm_Complexity}
Given the central path equations in~\eqref{central_path} with coefficients in $\mathbb{Z}$, Algorithm~\ref{analytic_param} computes a feasible $\rho$ using $2^{O(m+n^2)}$ arithmetic operations, where $\rho$ is bounded by $2^{2^{O(m+n^2)}}$.  
\end{theorem}

\subsection{Outlines of the procedures and proofs} We now briefly state the ideas behind our symbolic algorithm and the proof of our main results.

\vspace{5px}
\noindent
 In order to prove Theorem~\ref{thm_reparametrization}, we use degree bounds from the parameterized bounded algebraic sampling~\cite[Algorithm~12.18]{BPR06} and the quantifier elimination (see Theorem~\ref{14:the:tqe}), and then we apply the result of the Newton-Puiseux theorem (see Proposition~\ref{geometric_Puiseux}). 

In the next step, we elaborate on the proof technique of Theorem~\ref{thm_reparametrization} to develop a symbolic algorithm (Algorithm~\ref{analytic_param}) for an efficient computation of $\rho$. Our symbolic algorithm consists of the following basis elements in order:
\begin{itemize}
\item Computing a real univariate representation of the central path (Algorithm~\ref{alg:sampling});
\item Developing a formula which describes the graph of the $i^{\mathrm{th}}$ coordinate of the central path based on the real univariate representation;
\item Applying quantifier elimination to the preceding formula to develop a quantifier-free $\mathcal{P}_i$-formula;
\item Applying the univariate sign determination to identify $P_i \in \mathcal{P}_i$ whose zero set contains the graph of the $i^{\mathrm{th}}$ coordinate of the central path; 
\item Applying the Newton-Puiseux algorithm (Algorithm~\ref{analytic_param}) to $P_i$ to compute a feasible $\rho$.
\end{itemize}

 Algorithm~\ref{alg:sampling} invokes the parameterized bounded algebraic sampling ~\cite[Algorithm~12.18]{BPR06} and the quantifier elimination algorithm~\cite[Algorithm~14.5]{BPR06} to compute the real univariate representation of the central path for sufficiently small $\mu$, see Lemma~\ref{how_much_small}. The output is an $(\bar{n}+3)$-tuple of polynomials in $\mathbb{Z}[\mu,T]$ along with a Thom encoding $\sigma$ which describes the tail end of the central path, see Section~\ref{background_RAG}. By applying the quantifier elimination algorithm to a quantified formula derived from the output of Algorithm~\ref{alg:sampling} (i.e., real univariate representations, see~\eqref{central_solution_formula}), and then applying the univariate sign determination algorithm~\cite[Algorithm~10.13]{BPR06}, Algorithm~\ref{analytic_param} identifies a polynomial $P_i \in \mathbb{Z}[\mu,V_i]$, for $i=1,\ldots,\bar{n}$, whose zero set contains the graph of the $i^{\mathrm{th}}$ coordinate of the central path. Algorithm~\ref{analytic_param} invokes a symbolic Newton-Puiseux algorithm from~\cite[Algorithm~1]{W00} (see Proposition~\ref{Puiseux_Newton_Thm}) to compute ramification indices of all Puiseux expansions of $P_i = 0$ near $\mu = 0$, for every $i=1,\ldots,\bar{n}$, which converge to the $i^{\mathrm{th}}$ coordinate of the limit point of the central path. Here, we also utilize the real univariate representation of the limit point of the central path from~\cite[Algorithm~3.2]{BM22}. As a consequence, Algorithm~\ref{analytic_param} outputs a feasible $\rho$ by computing the least common multiple, over $i=1,\ldots,\bar{n}$, of the product of all distinct ramification indices corresponding to the above Puiseux expansions of $P_i = 0$. The reason for taking the ``product" of all ramification indices in Algorithm~\ref{analytic_param} will be made clear in Section~\ref{Rho_Computation}. 
 
 The proof of Theorems~\ref{Algorithm_Complexity} is determined based on the complexity of the parameterized bounded algebraic sampling, the quantifier elimination, and the symbolic Newton-Puiseux algorithm in~\cite[Algorithm~1]{W00}. Although Theorem~\ref{thm_reparametrization} gives a singly exponential upper bound on $\rho$, in Algorithm~\ref{analytic_param} we can only guarantee a doubly exponential upper bound on a feasible $\rho$, because we take the product of all ramification indices.

\begin{remark}
Notice that a feasible $\rho$ can be immediately derived using the degree of $P_i$ with respect to $V_i$, and without the use of Newton-Puiseux algorithm in Algorithm~\ref{analytic_param}. However, our goal here is to compute the best feasible $\rho$, if not optimal. This is also important for computational optimization purposes, because higher values of $\rho$ will result in ill-conditioning of the Jacobian matrix of the central path equations. 
\end{remark}

 \begin{remark}
We should indicate that Algorithm~\ref{analytic_param} will return the optimal $\rho$ when the branch containing the graph of the $i^{\mathrm{th}}$ coordinate of the central path is isolated from the other branches for every $i=1,\ldots,\bar{n}$. In particular, Algorithm~\ref{analytic_param} returns the optimal $\rho$ if $P_i$ for all $i=1,\ldots,\bar{n}$ are irreducible over $\mathbb{C}\{\mu\}$. 
\end{remark}

\begin{remark}
We should point out the paper~\cite{HLMZ21} in which the authors implement an analog of our analytic reparametrization to speed up the convergence for SDO using the homotopy continuation method. In contrast to~\cite{HLMZ21}, our paper provides a quantitative bound on the optimal $\rho$ and presents a symbolic algorithm to compute a feasible $\rho$.  
\end{remark}

The rest of this paper is organized as follows. In Section~\ref{survey}, we review prior results on the complexity of SDO, convergence, and analyticity of the central path in cases of linear optimization, linear complementarity problems, and SDO. In Section~\ref{AG_background}, we provide the preliminaries to real algebraic geometry, the theory of complex algebraic curves, and the central path. Our main results are presented in Sections~\ref{symbolic_algorithms} and~\ref{Newton_Polygon_Alg}. In Section~\ref{semi-algebraic_central_path}, we present the basis of Algorithm~\ref{alg:sampling} for the real univariate representation of the central path, when $\mu$ is sufficiently small. In Section~\ref{Puiseux_Expansion}, we explain our theoretical approach and prove Theorem~\ref{thm_reparametrization}. In Section~\ref{Newton_Polygon_Alg}, we present Algorithm~\ref{analytic_param} and then prove its complexity stated in Theorem~\ref{Algorithm_Complexity}. Finally, we end with concluding remarks and topics for future research in Section~\ref{conclusion}.

%%%%%%%%
%New Section
%%%%%%%%
\section{Prior and related work}\label{survey}
\subsection{Complexity}
The convex nature of SDO by no means implies polynomial solvability, in contrast to linear optimization. In the bit/real number model of computation, the complexity of SDO and polynomial optimization is well-known: there is no polynomial-time algorithm yet for an exact solution of these classes of optimization problems, see~\cite[Section~4.2]{B17} or~\cite{R97,RP96}. In the bit model of computation, a semi-definite feasibility problem either belongs to $\mathbf{NP} \cap \mathbf{co-NP}$ or $\mathbf{NP} \cup \mathbf{co-NP}$~\cite{R97}. In the real number model of computation~\cite{BCSS98}, a semi-definite feasibility problem belongs to $\mathbf{NP} \cap \mathbf{co-NP}$. Thus, the intrinsic nonlinearity of SDO, arising from the exponentially many positive semi-definite constraints, makes it no easier to solve than a polynomial optimization problem, which itself is, in general, NP-hard~\cite{M87}. In terms of algorithmic real algebraic geometry, the complexity of describing a primal-dual solution of SDO problem is 
\begin{align*}
\max\Big\{(2^n+m)^{O(n^2)},2^{O(mn)}\Big\},
\end{align*}
see~\cite[Algorithm~14.9]{BPR06}.

\subsection{Convergence}
On the computational optimization side, there exist efficient primal-dual IPM solvers to compute an approximate solution of SDO~\cite{Kl02,NN94}. However, even for an approximate solution, one can find well-structured pathological instances which IPM solvers either fail to solve or solve at a very slow convergence rate, see e.g.,~\cite{WW10}. By analogy with linear optimization, this poor performance can be linked to analytic or algebro-geometric properties of the central path~\cite{BL89a,BL89b,DSV12,DMS05,GS98,H02,LSZ98}. To mention but a few, there is a body of literature that deal with limiting behavior of the central path under the stronger condition of strict complementarity, see Definition~\ref{maximal_strict_comp}. Among other outstanding results, it is well-known that the central path is analytic at $\mu = 0$~\cite[Theorem~1]{H02}, and converges to the limit point $\big(X^{**},y^{**},S^{**}\big)$ at the rate of 1~\cite[Theorem~3.5]{LSZ98}: 
\begin{align}\label{Lips_Bounds}
 \|X(\mu)-X^{**}\| = O(\mu) \ \ \text{and} \  \ \|S(\mu)-S^{**}\| = O(\mu).
\end{align}
\noindent
This in turn accounts for the superlinear convergence of IPMs. On the other hand, both the analyticity at $\mu = 0$ and the Lipschitzian bounds~\eqref{Lips_Bounds} fail to hold in the absence of the strict complementarity condition~\cite{GS98}, see also Example~\ref{derivatives_diverge}. In~\cite{BM22}, the authors investigated the degree and the convergence rate of the central path from the perspective of algorithmic real algebraic geometry~\cite{BPR06}. The authors provided a lower bound on the convergence rate of the central path.
%%%%%%%%%
%New Proposition
%%%%%%%%%
\begin{proposition}[Theorem~1.1 in~\cite{BM22}]\label{distance_to_limit_point}
Let $(X^{**},y^{**},S^{**})$ be the limit point of the central path. Then for sufficiently small $\mu$ we have
\begin{align}\label{general_bounds} 
\|X({\mu})-X^{**}\|= O(\mu^{1/\gamma}) \ \ \text{and} \ \  \|S({\mu})-S^{**}\|=O(\mu^{1/\gamma}),
\end{align}
where $\gamma=2^{O(m+n^2)}$. 
\end{proposition}

\subsection{Analyticity}
Variants of Problem~\ref{CP_Analytic_Extension} have been also studied for linear optimization and linear complementarity problems, see e.g.,~\cite{G94,MT96,SW98,SW99}. For linear complementarity problems with no strictly complementary solution, the central path with reparametrization $\mu \mapsto \sqrt{\mu}$ can be analytically extended to $\mu = 0$~\cite{SW98,SW99}. This is mainly due to the fact that variables along the central path have magnitudes $O(1)$, $O(\mu)$, or $O(\sqrt{\mu})$~\cite{IPRT00}. However, such a classification does not necessarily hold for SDO, as shown in~\cite[Theorem~3.8]{MT19}. In fact, the only studies of Problem~\ref{CP_Analytic_Extension} for SDO are either under the assumption that a strictly complementary solution exists~\cite{H02,PS04} or under very restrictive conditions~\cite{NFOM05}.

%%%%%%%%
%New Section
%%%%%%%%

%%%%%%%%
%New Section
%%%%%%%%
\section{Background}\label{AG_background}
We briefly review the concepts of semi-algebraic sets, power and Puiseux series, complex algebraic curves, and the analyticity of the central path. Our notation for real closed fields, Puiseux series, formal power series, algebraic curves, and semi-definite optimization is consistent with those in~\cite{BM22,BPR06,F01,SWP08,W78}. For an exposition of algebraic curve theory and algebraic functions, the reader is referred to~\cite{BG91,JS87,W78}.

\vspace{5px}
\noindent
\paragraph{\textit{Definition of complexity}} 
%%sb changes
\hide{
In this paper, the complexity of an algorithm is defined according to~\cite[Chapter~8]{BPR06}, and it is a function of problem size, including the number of variables, the number of polynomials and degrees of polynomials. More specifically, the inputs to Algorithms~\ref{alg:sampling} to~\ref{Weierstrass_Alg} is an integral polynomial, see~\eqref{central_path_defining_polynomial}, formed by taking the sum of squares of polynomials in~\eqref{central_path}. By complexity we mean the number of arithmetic operations and comparisons in $\mathbb{Z}$ used by these algorithms. 
Since we deal with integral polynomials, this notion of complexity agrees with the bit (Turing) notion of complexity.
}
%%sb end of hide
By complexity of an algorithm we will mean the number of arithmetic operations in the ring
$\mathbb{Z}$ 
including comparisons needed by the algorithm (see ~\cite[Chapter~8]{BPR06}).
The complexity will be bounded in terms of the number of variables, the number of polynomials in the input and the degrees of polynomials. More specifically, the input to Algorithms~\ref{alg:sampling} and~\ref{analytic_param} is an integral polynomial, see~\eqref{central_path_defining_polynomial}, formed by taking the sum of squares of polynomials in~\eqref{central_path}. 
\hide{Sometimes in addition to counting the number of arithmetic operations we will be more precise and count the number of bit operations
used by the algorithm and bound it in terms of the bitsizes of the coefficients of the
polynomials in the input.
This last notion of complexity (defined in terms of the number of bits) agrees with the classical (Turing) notion of complexity.}
%%%%%%%%
%New Section
%%%%%%%%
%%sb
%%\subsection{Real closed fields and semi-algebraic sets}\label{background_RAG}
\subsection{Puiseux series, real closed fields and semi-algebraic sets}
\label{background_RAG}
%%sb adds
In this section we recall some relevant notions from real algebraic geometry
(the reader can consult~\cite[Chapter 2]{BPR06} for more details). From now on, for  an integral domain $\mathrm{D}$, we denote by $\mathrm{D}[Y_1,\ldots,Y_{\ell}]_{\leq d}$  the subset of polynomials in the ring $\mathrm{D}[Y_1,\ldots,Y_{\ell}]$ with degrees $\leq d$. Furthermore, $\mathrm{D}(Y_1,\ldots,Y_{\ell})$ denotes the  field of fractions of $\mathrm{D}[Y_1,\ldots,Y_{\ell}]$.

\begin{notation}
Let $\mathrm{K}$ be a field of characteristic zero. A polynomial $P \in \mathrm{K}[X]$ is called \textit{separable} if $\gcd(P,P') \in \mathrm{K}\setminus\{0\}$, where $\gcd(P,P')$ denotes the greatest common divisor of $P$ and its derivative $P'$. If there is no non-constant polynomial $A \in \mathrm{K}[X]$ such that $A^2$ divides $P$, then $P$ is called \textit{square-free}.
\end{notation}
If $\mathrm{K}$ is a field with characteristic zero, then $P$ is separable if and only if it is square-free.

\subsubsection{Puiseux series}
A \textit{Puiseux series} with coefficients in $\mathbb{C}$ (resp.  $\mathbb{R}$)
is a Laurent series with fractional exponents, i.e., a series of the form $\sum_{i = r}^{\infty} c_i \varepsilon^{i/q}$ where $c_i \in \mathbb{C}$ (resp. $c_i \in \mathbb{R})$, $i,r \in \mathbb{Z}$, and $q$ is a positive integer, which is called the \textit{ramification index} of the Puiseux series. 

Puiseux series appear naturally in algebraic geometry if we want to express the roots
of a polynomial $F(X,Y)$, parametrized by $X$ in a neighborhood of $0$. 
In this paper, they 
appear in the the proof of Theorem~\ref{thm_reparametrization}, 
%%a Puiseux series appear naturally 
when we express the roots of a polynomial $P_i \in \mathbb{Z}[\mu,V_i]$ near $\mu = 0$, where the zero set of $P_i$ contains the graph of the $i^{\mathrm{th}}$ coordinate of the central path.

The set of Puiseux series in $\varepsilon$ with coefficients in $\mathbb{C}$ (resp. $\mathbb{R}$) is a field, which we denote by 
$\mathbb{C}\langle \langle \varepsilon \rangle\rangle$ (resp. $\mathbb{R}\langle \langle \varepsilon \rangle\rangle$). It is a classical fact that the field 
$\mathbb{C}\langle \langle \varepsilon \rangle\rangle$ 
(resp. $\mathbb{R}\langle \langle \varepsilon \rangle\rangle$)
is algebraically closed (resp. real closed).

The subfield of $\mathbb{R}\langle \langle \varepsilon \rangle\rangle$ of elements which are algebraic over $\mathbb{R}(\varepsilon)$ is called the field of algebraic Puiseux series with coefficients in $\mathbb{R}$, 
and is denoted by $\mathbb{R} \langle \varepsilon \rangle$. 
It is the real closure of the ordered field $\mathbb{R}(\varepsilon)$ in which $\varepsilon$ is positive but smaller than every positive element of $\mathbb{R}$.
An alternative description of $\mathbb{R} \langle \varepsilon \rangle$ is that it is the field of germs of semi-algebraic functions to the right of the origin. Thus, each element
of $\mathbb{R} \langle \varepsilon \rangle$ is represented by a continuous semi-algebraic
function $(0,t_0) \rightarrow \mathbb{R}$ (see \cite[Chapter 2]{BPR06}), and this is the 
reason why the field $\mathbb{R} \langle \varepsilon \rangle$ plays an important role in the
study of germs of semi-algebraic curves (for example, the germ of the central path which is a semi-algebraic curve). 
%%sb hides
\hide{
The fields $\mathbb{R}$, $\mathbb{R}\langle \langle \varepsilon \rangle\rangle$, and $\mathbb{R} \langle \varepsilon \rangle$ are all examples \textit{real closed fields}~\cite[Exercise 2.19 and Theorem 2.91]{BPR06}.
Real closed fields have a unique order compatible with the field operations. For example,  the field $\mathbb{R} \langle \varepsilon \rangle$, this unique order is the one in which $\varepsilon$ is positive but smaller than every positive element of $\mathbb{R}$. 
}
We let $o(\cdot)$ denote the order of a Puiseux series, and it is defined as $o(\sum_{i = r}^{\infty} c_i \varepsilon^{i/q})=r/q$ if $c_{r} \neq 0$, see~\cite[Section~2.6]{BPR06}. 
We denote by $\mathbb{R}\langle\varepsilon\rangle_b$ the subring of $\mathbb{R}\langle\varepsilon\rangle$ of elements with are bounded over $\mathbb{R}$
(i.e. all Puiseux series in $\mathbb{R}\langle\varepsilon \rangle$ whose orders are
non-negative). 
We denote by $\lim_{\varepsilon}:\mathbb{R}\langle\varepsilon\rangle_b \rightarrow \mathbb{R} $ 
which maps a bounded Puiseux series $\sum_{i = 0}^{\infty} c_i\varepsilon^{i/q}$ to $c_0$
(i.e. to the value at $0$ of the continuous extension of the corresponding curve).

In terms of germs, the elements of $\mathbb{R}\langle\varepsilon\rangle_b$
are represented by semi-algebraic functions $(0,t_0) \rightarrow \mathbb{R}$ which can be extended continuously to $0$, and $\lim_\varepsilon$ maps such an element to the value
at $0$ of the continuous extension.

\subsubsection{Real closed fields}
While for the most part we will be concerned with the field of real numbers, at some points 
we will need to consider the non-Archimedean real closed extensions of $\mathbb{R}$ -- namely, the field of algebraic Puiseux series with coefficients in $\mathbb{R}$. Recall from~\cite[Chapter 2]{BPR06} that a real closed field $\R$ is an ordered field in which every positive element is a square and every polynomial having an odd degree has a root in $\R$.

Given a real closed field $\R$ and a set $\mathcal{P} \subset \R[Y_1,\ldots,Y_{\ell}]$, a \textit{quantifier-free $\mathcal{P}$-formula} $\mathrm{\Phi}(Y_1,\ldots,Y_{\ell})$ with coefficients in $\R$ is a Boolean combination of atoms $P > 0$, $P=0$, or $P < 0$ where $P \in \mathcal{P}$, and $\{Y_1,\ldots,Y_{\ell}\}$ are the \emph{free variables} of $\mathrm{\Phi}$. 
A quantified $\mathcal{P}$-formula is given by 
\begin{align*}
\Psi= (\mathrm{Q}_1 X_1) \cdots (\mathrm{Q}_{k} X_{k}) \ \mathcal{B}(X_1,\ldots, X_{k},Y_1,\ldots,Y_{\ell}),
\end{align*}
in which $\mathrm{Q}_i \in\{\forall,\exists\}$ are quantifiers and $\mathcal{B}$ is a quantifier-free $\mathcal{P}$-formula with $\mathcal{P} \subset \R[X_1,\ldots,X_{k},Y_1,\ldots,Y_{\ell}]$. A formula with no free variable is called a \textit{sentence}. The set of all $(y_1,\ldots,y_{\ell}) \in \R^{\ell}$ satisfying $\Psi$ is called the $\R$-realization of $\Psi$, and it is denoted by $\RR(\Psi,\R^{\ell})$. A $\mathcal{P}$-\textit{semi-algebraic} subset of $\R^{\ell}$ is defined as the $\R$-realization of a quantifier-free $\mathcal{P}$-formula.

It is a classical result due to Tarski \cite{Tarski51} that the first order theory of real closed fields is decidable and admits quantifier elimination. Thus every quantified formula
is equivalent modulo the theory of real closed fields to a quantifier-free formula. Later
in the paper we will use an effective version of this theorem~\cite[Theorem~14.16]{BPR06} equipped with complexity estimates, as stated below.

\begin{theorem}
[Quantifier Elimination]\label{14:the:tqe}
Let $\mathcal{P} \subset \R[X_{[1]},\ldots, X_{[\omega]}, Y]_{\leq d}$ be a
finite set of $s$ polynomials, where $X_{[i]}$ is a block of $k_i$ variables, and
$Y$ is a block of $\ell$ variables.
Consider the quantified formula 
\[
\Phi (Y) = (Q_1 X_{[1]})\cdots (Q_\omega X_{[\omega]}) \Psi(X_{[1]},\ldots, X_{[\omega]}, Y)
\]
and $\Psi$ a $\mathcal{P}$-formula.
 Then there exists a quantifier-free formula 
  \[ \Psi (Y) = \bigvee_{i=1}^{I} \bigwedge_{j=1}^{J_{i}} \Big(
     \bigvee_{n=1}^{N_{ij}} \sign (P_{ijn} (Y))= \sigma_{ijn} \Big)  \]
    equivalent to $\Phi$, where $P_{ijn} (Y)$ are polynomials in the variables $Y$, $\sigma_{ijn} \in
  \{0,1, - 1\}$, and
\begin{align*}
\sign(P_{ijn}(Y))\!&:=\begin{cases} \ \ 0 \ \ &P_{ijn}(Y) = 0,\\ \ \ 1 \ \ &P_{ijn}(Y) > 0,\\-1  &P_{ijn}(Y) < 0.  \end{cases}
\end{align*}
Furthermore, we have  
  \begin{eqnarray*}
    I & \leq & s^{(k_{\omega} +1) \cdots (k_{1} +1) ( \ell +1)} d^{O
    (k_{\omega} ) \cdots O (k_{1} ) O ( \ell )} ,\\
    J_{i} & \leq & 
      s^{(k_{\omega} +1) \cdots (k_{1} +1)} d^{O (k_{\omega} ) \cdots O (k_{1}
      )}
    ,\\
    N_{ij} & \leq &
      d^{O (k_{\omega} ) \cdots O (k_{1} )}
  ,
  \end{eqnarray*}
  and the degrees of the polynomials $P_{ijk} (y)$ are bounded by $d^{O
  (k_{\omega} ) \cdots O (k_{1} )}$. Moreover, there exists an algorithm (~\cite[Algorithm~14.5]{BPR06} (Quantifier Elimination)) to compute $\Phi (Y)$ with complexity 
  \begin{align*}
      s^{(k_{\omega} +1) \cdots (k_{1} +1)(\ell+1)} d^{O (k_{\omega} ) \cdots O (k_{1}) O(\ell)
      }
  \end{align*}
  in $\mathrm{D}$, where $\mathrm{D}$ denotes the ring generated by the coefficients of the polynomials in $\mathcal{P}$. If $\mathrm{D}=\mathbb{Z}$ and $\tau$ denotes an upper bound on bitsizes of $\mathcal{P}$, then the bitsizes of the integers in the intermediate computations and the output are bounded by $\tau d^{O(k_{\omega}) \cdots O (k_{1})O(\ell)}$.
\end{theorem}

\vspace{5px}
\noindent
We also state here the complexity of the quantifier elimination for deciding the truth or falsity of a sentence from~\cite[Theorem~14.14]{BPR06}.

\begin{theorem}[General Decision]\label{thm:general_decision}
Let $\mathcal{P} \subset \R[X_{[1]},\ldots, X_{[\omega]}]_{\leq d}$ be a
finite set of $s$ polynomials, where $X_{[i]}$ is a block of $k_i$ variables. Given a sentence $\Phi$, there exists an algorithm which decides the truth of $\Phi$ using  
  \begin{align*}
      s^{(k_{\omega} +1) \cdots (k_{1} +1)} d^{O (k_{\omega} ) \cdots O (k_{1})
      }
  \end{align*}
 arithmetic operations in $\mathrm{D}$, where $\mathrm{D}$ denotes the ring generated by the coefficients of the polynomials in $\mathcal{P}$.  
\end{theorem}

\subsubsection{Sign conditions, univariate representations, and Thom encodings}\label{sign_cond_Thom_encod}
In our algorithms we will need to represent points symbolically whose coordinates are algebraic over the ground field $\R$. We follow the representation used in the book~\cite{BPR06}.

Let $\R$ be a real closed field. Given a finite family $\mathcal{P} \subset \R[Y_1,\ldots,Y_{\ell}]$, a \textit{sign condition} on $\mathcal{P}$ is an element of $\{-1,0,1\}^{\mathcal{P}}$, i.e., a mapping $\mathcal{P} \to \{-1,0,1\}$. The \textit{realization of a sign condition $\sigma$ on a set $\mathcal{Z} \subset \R^{\ell}$} is defined as
\begin{align*}
\RR(\sigma,\mathcal{Z}):=\Big\{y \in \mathcal{Z} \mid \bigwedge_{P \in \mathcal{P}} \sign(P(y)) = \sigma(P)\Big\}. 
\end{align*}
If $\RR(\sigma,\mathcal{Z}) \neq \emptyset$, then $\sigma$ is said to be \textit{realized by $\mathcal{P}$ on $\mathcal{Z}$}. The set of all sign conditions realized by $\mathcal{P}$ on $\mathcal{Z}$ is denoted by $\SI(\mathcal{P},\mathcal{Z})$.

\vspace{5px}
\noindent
An \textit{$\ell$-univariate representation} is $(\ell+2)$-tuple of polynomials $u=\big(f, g_{0},\ldots,g_{\ell}\big) \in \R[T]^{\ell+2}$, where $f$ and $g_0$ are coprime. A \textit{real $\ell$-univariate representation} of an $x \in \R^{\ell}$ is a pair $(u,\sigma)$ of an $\ell$-univariate representation $u$ and a \textit{Thom encoding} $\sigma$ of a real root $t_{\sigma}$ of $f$ such that
\begin{align*}
x=\bigg(\frac{g_1(t_{\sigma})}{g_0(t_{\sigma})},\cdots,\frac{g_{\ell}(t_{\sigma})}{g_0(t_{\sigma})}\bigg) \in \R^{\ell}.
\end{align*}
Let $\Der(f):= \big\{f, f^{(1)},f^{(2)},\ldots,f^{(\deg(f))}\big\}$ denote a list of polynomials in which $f^{(i)}$ for $i > 0$ is the formal $i^{\mathrm{th}}$-order derivative of $f$ and $\deg(f)$ stands for the degree of $f$. The Thom encoding $\sigma$ of $t_{\sigma}$ is a sign condition on $\Der(f)$ such that $\sigma(f)=0$.

\vspace{5px}
\noindent
The notion of Thom encoding of real roots 
of a polynomial will be extensively used in this paper.
The following proposition~\cite[Proposition~2.27]{BPR06} indicates that for any
$P \in \R[X]$ and a root $x \in \R$ of $P$, the sign condition $\sigma$ realized by $\Der(P)$ at $x$ characterizes the root $x$. 

\begin{proposition} [Thom's Lemma] 
\label{prop:Thom}
Let $P \subset \R[X]$ be a univariate polynomial and $\sigma \in \{-1,0,1\}^{\Der(P)}$. Then the realization of the sign condition $\sigma$ is either empty, a point, or an open interval.
\end{proposition}
 
%%%%%%%%
%New Section
%%%%%%%%
\subsection{Local parametrization of complex algebraic curves}
%%sb adds
In this section, we recall basic definitions of affine and projective complex algebraic curves and some facts about their local parametrizations that will play a role later in the paper.

\subsubsection{Algebraic curves and local parametrization}
An \textit{algebraic} subset of $\mathbb{C}^{\ell}$ is the zero set of a set $\mathcal{P}$ of polynomials in $\mathbb{C}[Y_1,\ldots,Y_{\ell}]$.
%%sb
%%as follows   
We denote
\begin{align*}
\zero(\mathcal{P},\mathbb{C}^{\ell})\!:=\Big\{y \in \mathbb{C}^{\ell} \mid \bigwedge_{P \in \mathcal{P}} P(y)=0\Big\}.
\end{align*}
 As a special case, an \textit{affine complex algebraic curve} $\mathcal{C}$ is defined as the zero set of a bivariate polynomial $F \in \mathbb{C}[X,Y]$, i.e.,
\begin{align}\label{def:affine_curve}
\mathcal{C}:=\zero(F,\mathbb{C}^2)=\{(x, y) \in \mathbb{C}^2 \mid F(x, y) = 0\}.
\end{align}
%%sb not more generally
%%More generally, 
If $G \in \mathbb{C}[X,Y,Z]$ is a homogeneous polynomial 
%%(e.g., by homogenizing $F$), then we can define a \textit{projective complex algebraic curve} as the zero set of $G$ in the complex projective plane $\mathbb{P}^2(\mathbb{C})$~\cite{G89}:
we call 
the subset of $\mathbb{P}^2(\mathbb{C})$ 
\begin{align*}
\mathcal{D}:=\{(x:y:z) \in \mathbb{P}^2(\mathbb{C}) \mid G(x,y,z) = 0\}
\end{align*}
\textit{projective complex algebraic curve}.

Let $\mathbb{C}[[T]]$ be the ring of formal power series in $T$ over $\mathbb{C}$, and let $\mathbb{C}((T))$ denote its 
%%sb
%%quotient field, 
field of fractions.
%%sb omits
%%i.e. 
%%the field of \textit{formal Laurent series}.
%%sb
%%An element of $\mathbb{C}((T))$ is represented by 
Let $\phi=\sum_{i=r}^{\infty} c_i T^{i} \in \mathbb{C}((T))$, where $c_i \in \mathbb{C}$ and $r \in \mathbb{Z}$. If $c_r \neq 0$, the integer $r$ is defined to be the order of $\phi$, denoted by $o(\phi)=r$. If $\phi = 0$, we define $o(\phi)=\infty$. We call a \textit{projective local parametrization} of $\mathcal{D}$ a point $ (\phi):=\big(\phi_1(T): \phi_2(T):\phi_3(T)\big) \in \mathbb{P}^2\big(\mathbb{C}((T))\big)$ such that 
\begin{align*}
G\big(\phi_1(T), \phi_2(T),\phi_3(T)\big) = 0,
\end{align*}
and for every $\psi \neq 0$ it holds that $\psi \phi_i \not \in \mathbb{C}$ for some $i$. Given an affine algebraic curve $\mathcal{C}$ and a projective local parametrization for its projective closure $\mathcal{C}^*$ (defined by the homogenization of $F$ in ~\eqref{def:affine_curve}), we define an
\textit{affine local parametrization} of $\mathcal{C}$ as 
\begin{align*}
(\phi_1(T)/\phi_3(T),\phi_2(T)/\phi_3(T)) \in \mathbb{C}((T))^2.
\end{align*}
If $\min_{i \in \{1,2,3\}}\{o(\phi_i)\} = 0$, then the \textit{center} of a projective local parametrization is defined to be $(\phi_1(0):\phi_2(0):\phi_3(0))$. A local parametrization of $\mathcal{D}$ is called \textit{reducible} if $\phi_i \in \mathbb{C}((T^s))$ for some $s > 1$ and every $i=1,2,3$. Two local parametrizations $(\phi_1:\phi_2:\phi_3)$ and $(\psi_1:\psi_2:\psi_3)$ are called \textit{equivalent} if there exists a $\varphi \in \mathbb{C}((T))$ with $o(\varphi)=1$ such that $\phi_i = \psi_i (\varphi)$ for $i=1,2,3$. A \textit{place} of $\mathcal{C}$ (resp. $\mathcal{D}$) is the equivalence class of all its affine (resp. projective) irreducible local parametrizations.
%AM changes
\hide{
A \textit{place} of $\mathcal{D}$ is the equivalence class of all irreducible local parametrizations of $\mathcal{D}$. 
%%sb changes please check
%%A place of an affine algebraic curve is defined as the place of its projective closure.
A place of an affine algebraic curve $\mathcal{C}$ is a place of its projective closure 
whose center belongs to $\mathcal{C}$.}
The center of a place is the common center of its local parameterizations. 
%%sb changes pl check
\hide{
A place is equivalent to a \textit{branch} of $\mathcal{C}$, i.e., the set of all points $(\phi_1(t),\phi_2(t))$ for $t$ in a neighborhood of zero, where $\phi_1$ and $\phi_2$ are germs%
\footnote{A germ of holomorphic functions at $x_0$ means an equivalence class of all holomorphic functions which yield the same values in a neighborhood of $x_0$.} of holomorphic functions at zero. Equivalently, a branch of $\mathcal{C}$ is the zero set of an irreducible factor of $F$ over $\mathbb{C}\{X\}$.
}
%%sb pl check for accuracy.
A place is often referred to as a \textit{branch} of $\mathcal{C}$, i.e., the set of all points $(\phi_1(t),\phi_2(t)) \in \mathcal{C}$ for $t$ in a neighborhood of zero, where $\phi_1$ and $\phi_2$ are germs%
\footnote{A germ of holomorphic functions at $x_0$ means an equivalence class of all holomorphic functions which yield the same values in a neighborhood of $x_0$.} of holomorphic functions at zero.

\begin{example}\label{ex:local_param}
The affine complex algebraic curve (so-called cusp) defined by
\begin{align*}
    F_1(X,Y) = Y^2 - X^3
\end{align*}
has a single branch with a center at $(0,0)$, which is locally parameterized by $(t^2,t^3)$. On the other hand, the complex curve (so-called nodal cubic) defined by 
\begin{align*}
F_2(X,Y) = Y^2 - X^3 - X^2 = (Y-X\sqrt{X+1})(Y+X\sqrt{X+1})
\end{align*}
has 2 branches centered at $(0,0)$, which are locally parameterized by $(t, \phi_1(t))$ and $(t, \phi_2(t))$, see Figure~\ref{NodalCusp}. Notice that $\phi_1(t)$ and $\phi_2(t)$ are power series expansions of $X\sqrt{X+1}$ and $-X\sqrt{X+1}$, respectively. 
\begin{figure}
\centering
 \begin{minipage}[c]{0.5\textwidth}
\includegraphics[height=1.8in]{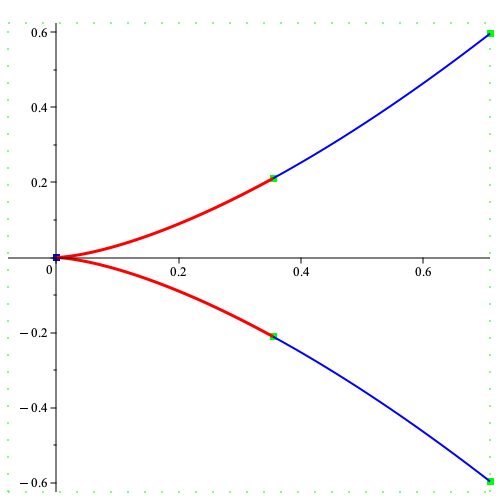}
\end{minipage}\hfill
\begin{minipage}[c]{0.5\textwidth}
\includegraphics[height=1.8in]{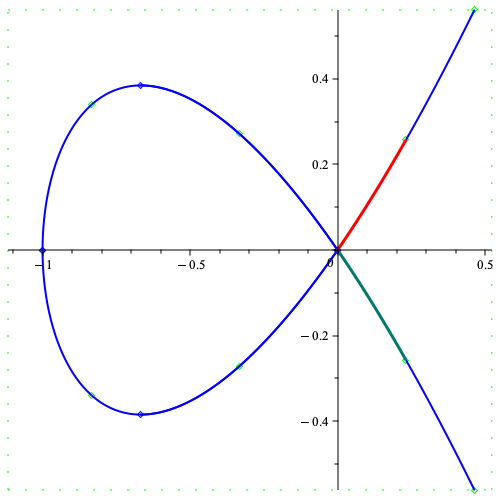}
\label{NodalCusp}
\end{minipage}
\caption{The branches of cusp (left) and nodal cubic (right) centered at $\mathbf{0}$.}
\end{figure}
\end{example}

We should also note that a point of a parametrization might be a center for more than one branch of $\mathcal{C}$.

\begin{example}[Example~2.66 in~\cite{SWP08}]\label{joint_center}
Consider the algebraic curve $\mathcal{C}$ defined by
\begin{align*}
F(X,Y)=Y^5-4Y^4+4Y^3+2X^2Y^2
-XY^2+2X^2Y+2XY+X^4+X^3.
\end{align*}
The curve $\mathcal{C}$ has two branches over $X = 0$ both centered at $\mathbf{0}$. The roots corresponding to these branches have ramification indices 1 and 2, see Figure~\ref{Double_Branches}.
 \begin{figure}
\includegraphics[height=2.0in]{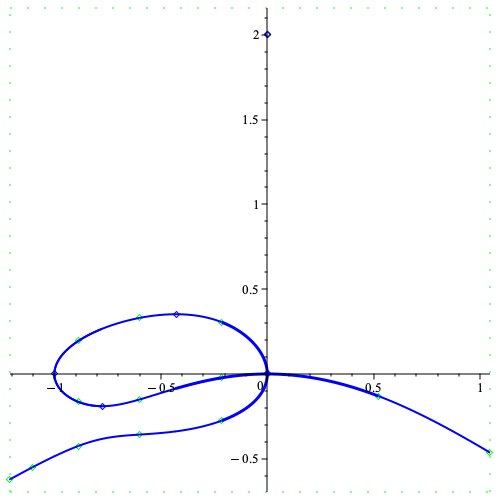}
\caption{In this example, $\mathbf{0}$ is the center of two branches.}
\label{Double_Branches}
\end{figure}
\end{example}

%%%%%%%%
%New Notation
%%%%%%%%
\begin{notation}
In this paper, a complex number is denoted by $a + b\sqrt{-1}$, where $a,b \in \mathbb{R}$, and $\exp(it)=\cos(t)+i\sin(t)$ is the complex exponential function. We use upper case letters for indeterminates of polynomials and Laurent and (convergent) power series, while lower case letters are used for the arguments of (holomorphic) functions.  
\end{notation}

\subsubsection{Newton-Puiseux theorem}
Let $\mathbb{C}\{T\}$ denote the ring of convergent power series in $T$, and assume without loss of generality that $(0,0) \in \mathcal{C}$. If $\partial F/\partial Y(0,0) \neq 0$, then the implicit function theorem provides an affine local parametrization $(T,\phi(T))$ of $F=0$ with center  $(0,0)$, where $\phi(T) \in \mathbb{C}\{T\}$. In general, the Newton-Puiseux theorem~\cite[Theorem~3.1 of Chapter IV]{W78} shows that an affine complex algebraic curve $\mathcal{C}$ has affine local parameterizations of the type $(T^q, \phi(T))$ for some $1 \le q \le d$, where $d:=\deg_Y(F)$. The proof of the Newton-Puiseux theorem is constructive and uses the Newton polygon to show that the field of Puiseux series $\mathbb{C}\langle\langle T \rangle\rangle  = \bigcup_ {s=1}^{\infty} \mathbb{C}((T^{1/s}))$ is algebraically closed.
%%%%%%%%%
%New Proposition
%%%%%%%%%
\begin{proposition}[Theorems 3.2 and~4.1 and Section~4.2 of Chapter IV in~\cite{W78}]\label{Puiseux_Newton_Thm}
Let 
\begin{align}\label{bivariate_polynomial}
F(X,Y)=a_0 + a_1Y + \cdots + a_d Y^d \in \mathbb{C}[X][Y],
\end{align}
where $a_d \neq 0$. There exists $d$ (not necessarily distinct) Puiseux series $\psi_i(X) \in \mathbb{C} \langle \langle X \rangle \rangle$ for $i=1,\ldots,d$ such that 
\begin{align*}
F(X,Y)=a_d \prod_{i=1}^d (Y - \psi_i).
\end{align*}
\noindent
There exist $k$ places of the curve $F = 0$ with center at $(0,0)$ corresponding to each $\psi_i$ with a positive order and multiplicity $k$. Conversely, corresponding to each place $(\phi_1,\phi_2)$ of $\mathcal{C}$ with center at $(0,0)$ there exist $o(\phi_1)$ roots of $F(X,Y)=0$ with identical positive orders.  
\end{proposition} 
\noindent
If $a_d(0) = 0$ in Proposition~\ref{Puiseux_Newton_Thm}, then $F(0,Y) = 0$ has less than $d$ roots (including multiplicity). In that case, there exist Puiseux series $\psi_i$ in Proposition~\ref{Puiseux_Newton_Thm} with negative orders, and they correspond to places at infinity.

\subsubsection{Weierstrass polynomial}
The constructive proof of the Newton-Puiseux theorem yields an iterative procedure to compute ramification indices of the Puiseux expansions in Proposition~\ref{Puiseux_Newton_Thm}, see Section~\ref{Newton_Puiseux_algorithm}. A specialized version of Proposition~\ref{Puiseux_Newton_Thm} can be obtained if $F$ is assumed to be an irreducible Weierstrass polynomial. A polynomial $W \in \mathbb{C} \{X\} [Y]$ is called a \textit{Weierstrass polynomial} of degree $d$ if
\begin{align*}
W=a_0 + a_1 Y + \cdots + a_{d-1} Y^{d-1} + Y^{d},
\end{align*}
where
$a_j(0)=0$ for $j=0,\ldots,d-1$. Given an irreducible $W$, the local parametrization of Proposition~\ref{Puiseux_Newton_Thm} can be explicitly written in terms of the degree of $W$. 
%%%%%%%%%
%New Proposition
%%%%%%%%%
\begin{proposition}[Section 7.8 in~\cite{F01}]\label{geometric_Puiseux}
Let $W$ be irreducible over $\mathbb{C} \{X\}$. Then $(T^d, \phi(T))$ is an affine local parametrization of $W = 0$, which determines a unique branch with center at $(0,0)$. Furthermore, $W = 0$ has $d$ distinct roots near $x = 0$, and they are all described by
\begin{align*}
 \psi_i (X)= \sum_{j = 1}^{\infty} c_{j} \Big(\exp\big(2\pi \sqrt{-1} (i-1)/d\big)X^{1/d}\Big)^{j}, \qquad i=1,\ldots,d. 
\end{align*}
\end{proposition}
%%%%%%%%
%New Example
%%%%%%%%
\begin{example}
The condition on the irreducibility of $W$ is an integral part of Proposition~\ref{geometric_Puiseux} and cannot be dropped. For instance, the reducible Weierstrass polynomial
\begin{align*}
 W(X,Y)=(Y^3-X^2)(Y^2-X^3)=Y^5-X^3Y^3-X^2Y^2+X^5 \in \mathbb{C}[X,Y]
 \end{align*}
 has roots $Y=X^{\frac23}$ and $Y=\pm X^{\frac32}$, while there is no local parametrization of the type $(T^5,\phi(T))$ around $x = 0$. 
 \end{example} 
\noindent
Let $\mathbb{C}\{X,Y\}$ denote the ring of convergent power series in $X$ and $Y$. Then Proposition~\ref{geometric_Puiseux} can be applied to any $F \in \mathbb{C}[X,Y]$ with $F(0,Y) \neq 0$. This fact immediately follows from the \textit{Weierstrass preparation theorem}.
%%%%%%%%%
%New Proposition
%%%%%%%%%
\begin{proposition}[Section~6.7 in~\cite{F01}]\label{WPT}
Suppose that $F \in \mathbb{C}\{X,Y\}$ such that 
\begin{align*}
F(0,Y) \neq 0 \ \ \text{and} \ \ o(F(0,Y)) = d.
\end{align*}
Then $F$ can be uniquely written as $F = U W$, where $U \in \mathbb{C} \{X,Y\}$ is a unit element and $W \in \mathbb{C} \{ X \}[Y]$ is a Weierstrass polynomial of degree $d$. In particular, if $F \in \mathbb{C} \{X\}[Y]$, then $U \in \mathbb{C} \{X\}[Y]$ holds as well. 
\end{proposition}

\begin{remark}\label{branch_alt_def}
Locally, a branch of $\mathcal{C}$ at a point (say $\mathbf{0} \in \mathcal{C}$) is the set of zeros of an irreducible factor of $F$ over $\mathbb{C}\{X\}$, see~\cite[Page~123]{F01}. For instance, $F_1 = 0$ in Example~\ref{ex:local_param} has one branch  
%%dsb changes
%%over $X = 0$ 
with center $\mathbf{0}$
because $F_1$ is irreducible over $\mathbb{C}\{X\}$. On the other hand, $F_2$ has two irreducible factors in $\mathbb{C}\{X\}[Y]$ and thus $F_2 = 0$ has two branches 
%%sb
%%over $X = 0$.
with center $\mathbf{0}$.
%%\textcolor{red}{I don't know if we want to keep the last sentence.} 
%%\textcolor{blue}{This is related Weierstrass preparation and Section 3.4}
\end{remark}

%%%%%%%%
%New Section
%%%%%%%%
\subsection{Analyticity of the central path}\label{CP_SDO}
The higher-order derivatives of the central path with respect to $\mu$ are all well-defined, by the  implicit function theorem. More concretely, the $i^{\mathrm{th}}$-order derivatives of the central path for $i \ge 1$ can be obtained by solving
\begin{align*}
\langle A^{i} , X^{(i)} \rangle&=0, \qquad  i=1,\ldots, m,\\[-1\jot]
\sum_{i=1}^m y^{(i)}_i A^i+S^{(i)}&=0,\\[-1\jot]
X^{(i)}S(\mu)+S(\mu)X^{(i)}+X(\mu)S^{(i)}+S^{(i)}X(\mu)&=\\
\begin{cases} I_n \ \ &i=1,\\
-\sum_{j=1}^{i-1} {i \choose j}\big(X^{(j)}S^{(i-j)}+S^{(j)}X^{(i-j)}\big), \ \ &i>1,
\end{cases}
\end{align*} 
where $(X^{(i)},y^{(i)},S^{(i)})$ denotes the $i^{\mathrm{th}}$-order derivative, and the coefficient matrix is always nonsingular.
The analyticity of the central path at $\mu = 0$ follows analogously if the Jacobian is nonsingular at the unique solution, see~\cite[Section~6]{GS98}, as a results of strict complementarity and nondegeneracy conditions, see~\cite[Theorem 3.1]{AHO98} and~\cite[Theorem 3.1]{H98}.

\vspace{5px}
\noindent
Obviously, the bounds~\eqref{Lips_Bounds} and~\eqref{general_bounds} do not imply the boundedness of the derivatives as $\mu \downarrow 0$, see e.g., Example~\ref{derivatives_diverge}. In fact, the convergence of derivatives (of any order) is guaranteed only in the presence of the strict complementarity condition~\cite[Theorem~1]{H02}. For instance, consider an orthogonal basis $Q:=\big(Q_{\mathcal{B}},Q_{\mathcal{T}}, Q_{\mathcal{N}}\big)$ for the 3-tuple of mutually orthogonal subspaces%
\footnote{The subspaces $\column(X^{**})$ and $\column(S^{**})$ are orthogonal by the complementarity condition $XS=0$.} 
\begin{align*}
\Big(\column(X^{**}),\column(S^{**}),(\column(X^{**})+\column(S^{**}))^{\perp}\Big),
\end{align*}
where $\column(\cdot)$ denotes the column space of a matrix. If $Q_{\mathcal{T}} \neq \{\mathbf{0}\}$, then the first-order derivative of the central path fails to converge, because
\begin{align*}
Q^T_{\mathcal{T}} X^{(1)}(\mu) Q_{\mathcal{T}} Q^T_{\mathcal{T}}S(\mu) Q_{\mathcal{T}} + Q^T_{\mathcal{T}}X(\mu) Q_{\mathcal{T}}Q^T_{\mathcal{T}}S^{(1)}(\mu)Q_{\mathcal{T}}= I_{n_{\mathcal{T}}},
\end{align*}
while both $Q^T_{\mathcal{T}}S(\mu) Q_{\mathcal{T}} \to 0$ and $Q^T_{\mathcal{T}}X(\mu) Q_{\mathcal{T}} \to 0$ as $\mu \downarrow 0$
\hide{
\footnote{There exists a lower bound $2^{1-n}$ on the convergence rates of $Q^T_{\mathcal{T}} X(\mu) Q_{\mathcal{T}}$ and $Q^T_{\mathcal{T}} S(\mu) Q_{\mathcal{T}}$~\cite[Theorem~3.8]{MT19}.}}. All this indicates that the central path has no analytic extension to $\mu=0$.
%%%%%%%%
%New Example
%%%%%%%%
\begin{example}\label{derivatives_diverge}
The minimization of a linear objective function over the \emph{3-elliptope}, see Figure~\ref{3elliptope_fails}, can be cast into a SDO problem:
\begin{align}\label{feasible_3elliptope}
\min\bigg\{4x-4y-2z \mid 
\begingroup 
\setlength{\tabcolsep}{.75pt} % Default value: 6pt
\renewcommand{\arraystretch}{.8}
\begin{pmatrix} 1 & \ x & \ y\\x & \ 1 & \ z\\y & \ z & \ 1 \end{pmatrix}
\endgroup \succeq 0\bigg\}.
\end{align} 
\noindent
The unique solution of~\eqref{feasible_3elliptope} is given by
\begin{align}\label{unique_optimal_solution}
\setlength{\arraycolsep}{1pt}
X^*=\begin{pmatrix} \ \ 1 & \ -1 & \ \ \ 1 \\ -1 & \ \ \ 1 & \ -1\\ \ \ 1 & \ -1 & \ \ \ 1 \end{pmatrix}, \ \
y^*=(-4, \ -1, \ -1)^T,\ \
S^*=\begin{pmatrix} \ \ 4 & \ \ \ 2 & \ -2 \\ \ \ 2 & \ \ \ 1 & \ -1\\ -2 & \ -1 & \ \ \ 1 \end{pmatrix},
\end{align}
which is not strictly complementary. 

The graph of $X_{12}(\mu)$ can be described as the set of all $(\mu,t)$ with $\mu > 0$ satisfying 
\begin{align}\label{central_path_defining_equation}
F(\mu,T):=2T^3+(2-\mu/2)T^2-(\mu+2)T-2=0,
\end{align}
\begin{align*}
\begingroup 
\setlength{\tabcolsep}{.75pt} % Default value: 6pt
\renewcommand{\arraystretch}{.8}
\begin{pmatrix} 1  & T & -T\\T & 1 & -T^2+\mu T/2+1\\-T & -T+\mu T/2+1 & 1 \end{pmatrix} \succ 0,
\endgroup
\\ 
\begingroup 
\setlength{\tabcolsep}{.75pt} % Default value: 6pt
\renewcommand{\arraystretch}{.8}
\begin{pmatrix} \mu-4T & 2 & -2\\2 & -2/T-1 & -1 \\-2 & -1 & -2/T-1 \end{pmatrix} \succ 0.
\endgroup
\end{align*}
By the uniqueness of~\eqref{unique_optimal_solution}, we must have $X_{12}(\mu) \to -1$ as $\mu \downarrow 0$. Then it follows from~\eqref{central_path_defining_equation} that 
\begin{align*}
X^{(1)}_{12}(\mu)=\frac{X^2_{12}(\mu)/2+X_{12}(\mu)}{6X^2_{12}(\mu)+(4-\mu)X_{12}(\mu)-(\mu+2)} \to \infty
\end{align*}
as $\mu \downarrow 0$. This explains the tangential convergence of the central path in Figure~\ref{3elliptope_fails}.  

 \begin{figure}
\includegraphics[height=1.5in]{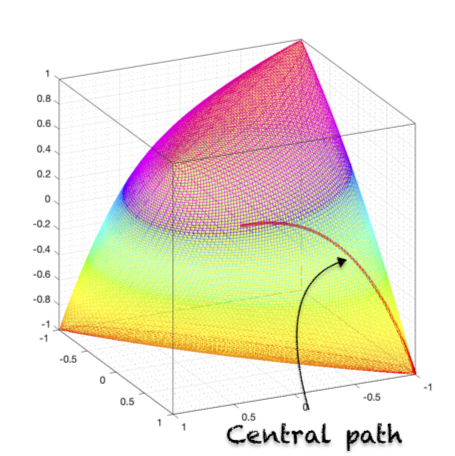}
\caption{The derivatives of the central path fail to exist at $\mu = 0$.}
\label{3elliptope_fails}
\end{figure}
\end{example}
\noindent
The algebraic curve $F = 0$ in Example~\ref{derivatives_diverge} has two branches over $\mu = 0$ centered at $(0,-1)$ and $(0,1)$, as demonstrated by Figure~\ref{3elliptope_alg_curve}.
\hide{
Since $\partial{F}/\partial \mu(0,-1) = 0$, the condition of the implicit function theorem fails, which makes $(0,-1)$ a branch point of the curve.}
By invoking the ``Algcurves" package in Maple%
\footnote{Available at \url{https://www.maplesoft.com/}} we can numerically compute the Puiseux expansions of all roots of $F = 0$ near $\mu = 0$ as follows
\begin{align*}
T_{1}(\mu)&=-1+\frac{\sqrt{8}}{8} \mu^{\frac12}+\frac{1}{32} \mu -\frac{11\sqrt{8}}{2048} \mu^{\frac32}-\frac{3}{512} \mu^{2}-\frac{121\sqrt{8}}{1048576} \mu^{\frac52}\\
&\quad+\frac{15}{32768} \mu^3+\frac{19405\sqrt{8}}{268435456}\mu^{\frac72}+ \cdots, \\
T_{2}(\mu)&=-1-\frac{\sqrt{8}}{8} \mu^{\frac12}+\frac{1}{32} \mu +\frac{11\sqrt{8}}{2048} \mu^{\frac32}-\frac{3}{512} \mu^{2}+\frac{121\sqrt{8}}{1048576} \mu^{\frac52}\\
&\quad+\frac{15}{32768} \mu^3-\frac{19405\sqrt{8}}{268435456}\mu^{\frac72}+ \cdots ,\\
T_{3}(\mu)&=1+\frac{3}{16} \mu+\frac{3}{256} \mu^{2} -\frac{15}{16384} \mu^{3}-\frac{15}{262144}\mu^4+\frac{345}{16777216} \mu^5\\
&\quad-\frac{21}{33554432}\mu^6-\frac{1869}{4294967296}\mu^7+\cdots\hfill ,
\end{align*}
where $T_1$ is the Puiseux expansion of the $X_{12}$ coordinate of the central path.
\hide{$T_2$ corresponds to an exterior semi-algebraic path converging to the same limit point, and $T_3$ is a semi-algebraic path converging to an infeasible vector.}
The order of the Puiseux series $T_1$ indicates the non-Lipschitzian convergence 
\begin{align*}
\|X(\mu)-X^{**}\| = O(\sqrt{\mu}) \ \ \text{and} \ \ \|S(\mu)-S^{**}\| = O(\sqrt{\mu}).
\end{align*}
\begin{figure}
\center
\begin{minipage}[c]{0.5\textwidth}
\includegraphics[height=2.2in]{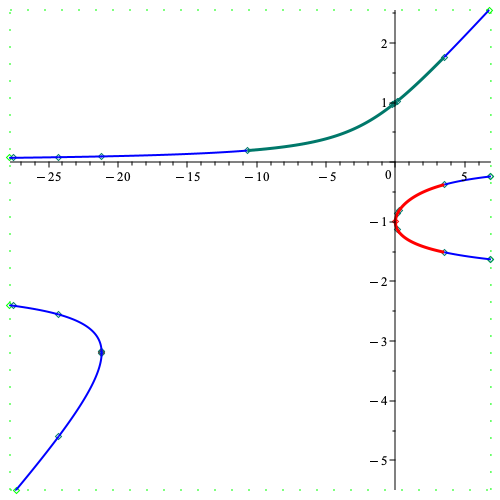}
\end{minipage}
\begin{minipage}[c]{0.4\textwidth}
\includegraphics[height=1.5in]{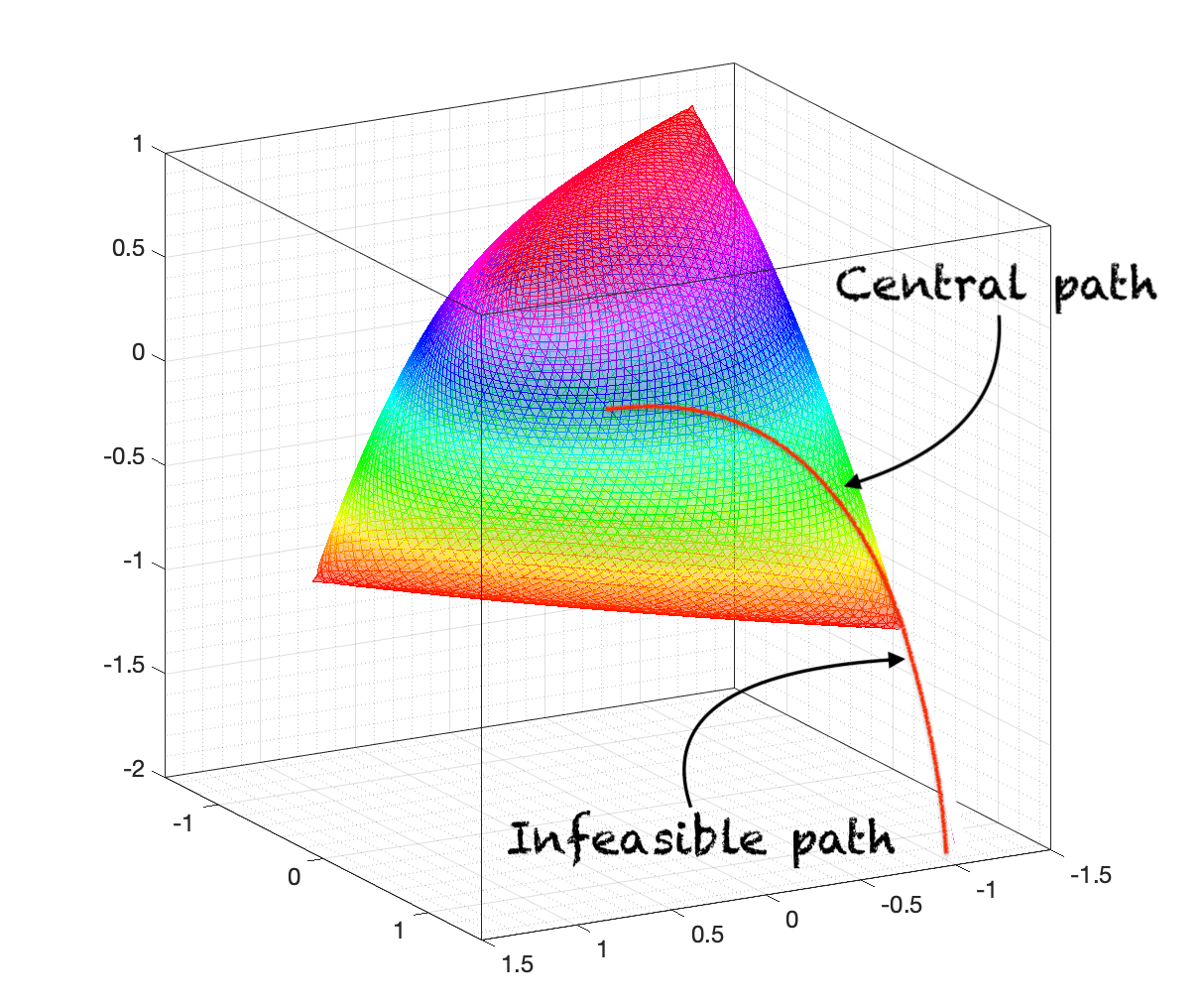}
\end{minipage}
\caption{$T_1(\mu)$ (the upper red segment of the branch) is the expansion of $X_{12}(\mu)$; $T_2(\mu)$ (the lower red segment of the branch) is the expansion of a semi-algebraic path converging to $X^{**}_{12}$ from the exterior of the positive semi-definite cone; $T_3(\mu)$ (the green segment) is the expansion of a semi-algebraic path converging to an infeasible value.}
\label{3elliptope_alg_curve}
\end{figure}

\hide{
%%%%%%%%
%New Remark
%%%%%%%%
\begin{remark}
The ramification indices in Example~\ref{derivatives_diverge} can be also read off from the monodromy group of $F = 0$. Using the Algcurves package of Maple, we can numerically locate the branch points of $F = 0$ as $\mu=-1.41253 \pm 4.71 \sqrt{-1}$, $\mu-21.175$, and $\mu=0$ with cycles $(1,3)$, $(1,2)$, $(1,2)$, respectively. As with the branch point $\mu = 0$, all this means that the analytic continuation of a holomorphic solution $\phi_1(\mu)$ near $\mu = 0$ results in $\phi_2(\mu)$ when complex $\mu$ wraps once around the branch point $\mu = 0$. Thus, the cycle length 2 means the ramification index $q = 2$. 
\end{remark}}

%%%%%%%%
%New Section
%%%%%%%%
\section{Reparametrization of the central path}\label{symbolic_algorithms}
Although the Lipschitzian bounds~\eqref{Lips_Bounds} fail to exist in the absence of the strict complementarity condition, we can still exploit the local information around the center point to recover the analyticity of the central path. This remedial action can be applied to Example~\ref{derivatives_diverge}, where the ramification index of $T_1$ suggests the local parametrization $(\mu^{2},\Phi(\mu))$ of the curve $F = 0$ around $\mu = 0$, where 
\begin{align*}
\Phi(\mu)=-1+\frac{\sqrt{8}}{8} \mu+\frac{1}{32} \mu^2 -\frac{11\sqrt{8}}{2048} \mu^{3}-\frac{3}{512} \mu^{4}-\frac{121\sqrt{8}}{1048576} \mu^{5}+\frac{15}{32768} \mu^6+ \ldots, 
\end{align*}
see Figure~\ref{3elliptope_reparam}. Thus, one may adopt a reparametrization $\mu \mapsto \mu^{\rho}$ under which the central path is analytic at $\mu = 0$, where $\rho$ is a positive integer multiple of the ramification index of the Puiseux expansion. 

In the worst-case scenario, the magnitude of $q$ depends exponentially on $n$.
 \begin{figure}
\includegraphics[height=2.0in]{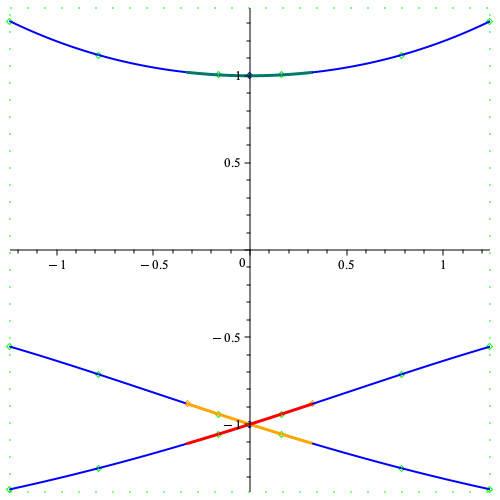}
\caption{The algebraic curve~\eqref{central_path_defining_equation} after the reparametrization $\mu \to \mu^2$. The reparametrized central path (red segment) is analytic at $\mu = 0$.}
\label{3elliptope_reparam}
\end{figure}
%%%%%%%%
%New Example
%%%%%%%%
\begin{example}[Example~3.3 in~\cite{Kl02}]\label{lower_bound_analyticity}
Consider the following SDO problem in dual form $(\mathrm{D})$:
\begin{align*}
\max\Bigg\{-y_n \mid S=
\begingroup 
\setlength{\tabcolsep}{.75pt} % Default value: 6pt
\renewcommand{\arraystretch}{.8}
\begin{pmatrix} 1 & y_1 & y_2 &\ldots & y_{n-1}\\y_1 & y_2 & 0 & \ldots & 0\\y_2 & 0 & y_3 & \ddots & \vdots\\ \vdots & \vdots & \ddots & \ddots & 0\\y_{n-1} & 0 & \ldots & 0 & y_n \end{pmatrix}
\endgroup \succeq 0\Bigg\},
\end{align*}
which has a unique solution $(y^{**},S^{**})$ with $y^{**}_i=0$ for $i=1,\ldots,n$. In this case, $y_2(\mu)=O(\mu^{2^{-(n-2)}})$, which implies a reparametrization $\mu \mapsto \mu^{q}$ with $q\ge 2^{n-2}$. 
\end{example} 
\noindent
In view of Example~\ref{derivatives_diverge}, our goal is to adopt a semi-algebraic approach to identify bivariate polynomials, analogous to $F(\mu,T)$ in~\eqref{central_path_defining_equation}, that describe the tail end of the central path in a coordinate-wise manner. All we need then are the ramification indices of the Puiseux expansions corresponding to the roots of these bivariate polynomials around $\mu = 0$, which give rise to a feasible $\rho$ for Problem~\ref{CP_Analytic_Extension}.

The basis elements of our semi-algebraic approach  are the real univariate representation of the central path for sufficiently small $\mu$, the quantifier elimination (see Theorem~\ref{14:the:tqe}), and the Newton-Puiseux theorem (see Proposition~\ref{Puiseux_Newton_Thm}), which we elaborate on in Sections~\ref{semi-algebraic_central_path} and~\ref{Puiseux_Expansion}.

%%%%%%%%
%New Section
%%%%%%%%
\subsection{Real univariate representation of the central path}\label{semi-algebraic_central_path}
The central path system~\eqref{central_path} can be considered as a $\mu$-infinitesimally deformed polynomial system $\mathcal{G} \subset \mathbb{Z}[\mu][V_1,\ldots,V_{\bar{n}}]$, i.e., a $\mu$-infinitesimal deformation of the polynomial system
\begin{align*}
\big\{\mathbf{A} \vectorize(X)&=b, \ \mathbf{A}^T y + \vectorize(S-C)=0, \ \vectorize(XS)=0\big\},
\end{align*}
whose zeros belong to $\mathbb{C}\langle\mu\rangle^{\bar{n}}$. Without loss of generality, we consider the restriction of the central path to the interval $(0,1]$. Recall that the central path is uniformly bounded, i.e., there exists, see Lemma~\ref{central_path_ball}, a rational $\varepsilon > $ such that 
\begin{align}\label{bound_on_algebraic_set}
\|(X({\mu}),y({\mu}),S({\mu}))\| \le 1/\varepsilon, \qquad \forall \mu \in (0,1].
\end{align} 
\noindent
Since we are interested in the central path which is bounded over $(0,1]$, we only characterize the bounded zeros of $\mathcal{G}$ in $\mathbb{R}\langle \mu \rangle^{\bar{n}}$, which are defined as
\begin{align*}
\zero_b\big(\mathcal{G},\mathbb{R}\langle \mu \rangle^{\bar{n}}\big)\!:=\zero\big(\mathcal{G},\mathbb{R}\langle \mu \rangle^{\bar{n}}\big) \cap \mathbb{R}\langle \mu \rangle^{\bar{n}}_b,
\end{align*}
where $\mathbb{R}\langle \mu \rangle^{\bar{n}}_b$ denotes the subring of 
$\mathbb{R}\langle \mu \rangle$ consisting of elements
which are bounded over $\mathbb{R}$. The central path for sufficiently small positive $\mu$ is a bounded solution of $\zero\big(\mathcal{G},\mathbb{R}\langle \mu \rangle^{\bar{n}}\big)$. Therefore, the limit point $(X^{**},y^{**},S^{**})$ is contained in $\lim_{\mu} \zero_b\big(\mathcal{G},\mathbb{R}\langle \mu \rangle^{\bar{n}}\big)$.
%%sb
%%where $\lim_{\mu}$ is a ring homomorphism from $\mathbb{R}\langle \mu \rangle_b$ to $\mathbb{R}$.

\subsubsection{Parameterized bounded algebraic sampling}
Our approach to characterize the bounded zeros of $\mathcal{G}$ is to compute real univariate representation of the central path when $\mu$ is sufficiently small; this describes the coordinate $v_i(\mu)$ of the central path as a rational function of $\mu$ and the roots of a univariate polynomial. To that end, we define the two polynomials $Q \in \mathbb{Z} [\mu, V_1,\ldots, V_{\bar{n}}]$ and $\tilde{Q} \in \mathbb{Z} [\mu, V_1,\ldots, V_{\bar{n}+1}]$ as follows 
\begin{align}
Q&:= \|\mathbf{A}x-b\|^2 + \|\mathbf{A}^Ty+s-c\|^2+ \|XS-\mu I_n\|^2, \nonumber \\
\tilde{Q}&:=Q^2 + \big(\varepsilon^2(V_1^2+\ldots+V_{\bar{n}+1}^2)-1\big)^2, \label{central_path_defining_polynomial}
\end{align}
where $\varepsilon$ is defined in~\eqref{bound_on_algebraic_set}. Notice that for every fixed $\mu \in (0,1]$, $\zero(\tilde{Q}({\mu}),\mathbb{R}^{\bar{n}+1})$ is nonempty (it contains a central solution), and 
\begin{align*}
\bigcup_{\mu \in \mathbb{R}} \zero(\tilde{Q}({\mu}),\mathbb{R}^{\bar{n}+1})
\end{align*}
is bounded, because $\zero(\tilde{Q}({\mu}),\mathbb{R}^{\bar{n}+1})$ is the intersection of the cylinder based on $\zero(\mathcal{P}_{\mu},\mathbb{R}^{\bar{n}})$ and an $\bar{n}$-sphere, where 
\begin{align*}
\mathcal{P}_{\mu}:=\big\{\mathbf{A} \vectorize(X)&=b, \ \mathbf{A}^T y + \vectorize(S-C)=0, \ \vectorize(XS-\mu I)=0\big\}. 
\end{align*}
\noindent
\hide{
The idea here is to utilize the parameterized bounded algebraic sampling algorithm~\cite[Algorithm~12.18]{BPR06} with input $\tilde{Q}$ to describe a finite set of \textit{$X_1$-pseudo critical points}~\cite[Definition~12.41]{BPR06} on $\zero\big(\tilde{Q}(\mu),\mathbb{R}^{\bar{n}+1}\big)$, which meets every connected component of $\zero\big(\tilde{Q}(\mu),\mathbb{R}^{\bar{n}+1}\big)$~\cite[Proposition~12.42]{BPR06}, see also~\cite[Section~3.1]{BM22}.
}
The idea here is to utilize the parameterized bounded algebraic sampling algorithm~\cite[Algorithm~12.18]{BPR06} with input $\tilde{Q}$ to describe a finite set of sample points, which for every $\mu \in \mathbb{R}$ meets every connected component of $\zero\big(\tilde{Q}(\mu),\mathbb{R}^{\bar{n}+1}\big)$, see also ~\cite[Proposition~12.42]{BPR06}. The description of these sample points are given by a set $\mathcal{U}$ of parameterized univariate representations
\begin{align*}
u:=(f,g)=\big(f,(g_0,g_{1},\ldots, g_{\bar{n}+1})\big) \in \mathbb{Z}[\mu,T]^{\bar{n}+3}.
\end{align*}
We will prove in Lemma~\ref{how_much_small} that for sufficiently small positive $\mu$, there exists a univariate representation $u$ and a real root $t_{\sigma}$ of $f$ with Thom encoding $\sigma$ such that   
\begin{align}\label{central_path_representation}
v_{i}(\mu)&=\frac{g_{i}(\mu, t_{\sigma})}{g_{0}(\mu, t_{\sigma})} \in \mathbb{R}, \qquad i=1\ldots,\bar{n}, 
\end{align}
where $g_0(\mu, t_{\sigma}) \neq 0$. At the end, the problem of choosing the right $\big((f,g),\sigma\big)$ is a real algebraic geometry problem and can be decided by the quantifier elimination algorithm.

To see this, let $(\hat{X}(\mu),\hat{y}(\mu),\hat{S}(\mu))$ be associated to $((f,g),\sigma)$, and let, without loss of generality, $C_x,C_s \in \mathbb{Z}[\mu] [T,\Lambda]$ be the characteristic polynomials of $\hat{X}(\mu)$ and $\hat{S}(\mu)$, respectively. Then $((f,g),\sigma)$ represents the central path, when $\mu$ is sufficiently small, if the following two $\mathcal{Q}$-sentences with $\mathcal{Q} \subset \mathbb{Z}[\mu][T,\Lambda]$ are both true:
\begin{equation}\label{central_path_selection}
\begin{aligned}
(\exists T) (\forall \Lambda) \  \big(\sign(f^{(j)})=\sigma(f^{(j)}), \ j \in \mathbb{Z}_{\ge 0}\big) 
\wedge \big(\neg (C_x(T,\Lambda) = 0) \vee (\Lambda > 0)\big),\\
(\exists T) (\forall \Lambda) \ \big(\sign(f^{(j)})=\sigma(f^{(j)}), \ j \in \mathbb{Z}_{\ge 0}\big) 
\wedge \big(\neg (C_s(T,\Lambda) = 0) \vee (\Lambda > 0)\big).
\end{aligned}
\end{equation}
Now we can apply ~\cite[Algorithm~14.3]{BPR06} (General Decision) to~\eqref{central_path_selection}.

\vspace{5px}
\noindent
Algorithm~\ref{alg:sampling} summarizes our symbolic procedure for the real univariate representation of the central path.
%%%%%%%%%
%New Algorithm
%%%%%%%%%
\begin{algorithm}
\caption{Real univariate representation of the central path}
\label{alg:sampling}
\begin{algorithmic}
\State Input: The polynomial $\tilde{Q} \in \mathbb{Z} [\mu, V_1,\ldots,V_{\bar{n}+1}]$. \\
\Return The set $\mathcal{U}$ of parametrized univariate representations $u$ and Thom encodings $\sigma$ of roots of $f\in \mathbb{Z}[\mu][T]$; There exists a real univariate representation $(u,\sigma)$ with $u \in \mathcal{U}$ which describes the central path for sufficiently small positive $\mu$, see Lemma~\ref{how_much_small}.

\vspace{5px}
\State \textbf{Procedure:}
\begin{enumerate}
\item\label{param_sampling_Step} Generate the set $\mathcal{U}$ of parameterized univariate representations by applying~\cite[Algorithm~12.18]{BPR06} (Parameterized Bounded Algebraic Sampling) with input $\tilde{Q}$ and parameter $\mu$.
\item\label{Thom_encod_Step} Compute the ordered list of Thom encodings of the roots of $f$ in $\mathbb{R}\langle \mu \rangle$ by applying~\cite[Algorithm~10.14]{BPR06} (Thom Encoding) to each $f \in \mathbb{Z}[\mu][T]$ in the set $\mathcal{U}$.

\item\label{quantifier_elim} Decide which $((f,g),\sigma)$ describes the central path for sufficiently small positive $\mu$ by applying~\cite[Algorithm~14.3]{BPR06} (General Decision) to~\eqref{central_path_selection}.
\end{enumerate}
\end{algorithmic}
\end{algorithm}

\subsubsection{Complexity and the proof of correctness} The correctness of Algorithm~\ref{alg:sampling} follows from  Theorem~\ref{thm:general_decision}, and Lemma~\ref{how_much_small} below, and its complexity follows from the complexity of~\cite[Algorithm~12.18]{BPR06} and ~\cite[Algorithm~14.3]{BPR06}. First, we need the following quantitative result. 

%%%%%%%%
%New Lemma
%%%%%%%%
\begin{lemma}\label{central_path_ball}
The polynomial $\tilde{Q}$ in~\eqref{central_path_defining_polynomial} has coefficients with bitsizes bounded by $\tau 2^{O(m+n^2)}$.
\end{lemma}
\begin{proof}
All we need here is the magnitude of $1/\varepsilon$ in $\tilde{Q}$, which is also an upper bound on the norm of central solutions for all $\mu \in (0,1]$, see~\eqref{bound_on_algebraic_set}. From the central path equations in~\eqref{central_path} it follows that $\langle X(\mu)-X(1), S(\mu)-S(1) \rangle = 0$, which results in
\begin{align*}
\langle X(\mu),S(1) \rangle + \langle S(\mu),X(1) \rangle = n(\mu+1).
\end{align*}
Since the central solutions are positive definite, 
\begin{align*}
\langle X(\mu),S(1) \rangle >0, \ \langle S(\mu),X(1) \rangle >0,
\end{align*}
and thus for all $\mu \in (0,1]$ we have
\begin{align*}
\langle X(\mu),S(1) \rangle \le  2n \quad \text{and} \quad  \langle S(\mu),X(1) \rangle \le  2n.
\end{align*}
Furthermore, the centrality condition $XS = I$ implies that
\begin{equation}
\begin{aligned}\label{bound_on_Xu}
\|X(\mu)\| &\le \frac{2 n}{\lambda_{\min}(S(1))} \le 2n\lambda_{\max}(X(1)),\\
\|S(\mu)\| &\le \frac{2 n}{\lambda_{\min}(X(1))} \le 2n\lambda_{\max}(S(1)), 
\end{aligned}
\qquad\qquad \forall \mu \in (0,1].
\end{equation}
By the integrality of the data, see Assumption~\ref{integral_data}, there exists~\cite[Theorem~1]{BR10} a ball of radius $r=2^{\tau 2^{O(m+n^2)}}$ containing every isolated point of $\zero(\mathcal{P}',\mathbb{R}^{\bar{n}})$, including $(X(1),y(1),S(1))$, where
\begin{align*}
\mathcal{P}':=\big\{\mathbf{A} \vectorize(X)=b, \mathbf{A}^T y + \vectorize(S-C)=0, \vectorize(XS-I_n)=0\big\}.
\end{align*}
This also gives an upper bound on $\lambda_{\max}(X(1))$ and $\lambda_{\max}(S(1))$ which, by~\eqref{bound_on_Xu}, yields
\begin{align*}
\|X(\mu)\|,\|S(\mu)\| = 2^{\tau2^{O(m+n^2)}}, \qquad \forall \mu \in (0,1]. 
\end{align*}
Now the result follows when we choose $1/\varepsilon = 2^{\tau2^{O(m+n^2)}}$.  
\end{proof}
%%%%%%%
%New Lemma
%%%%%%%
\begin{lemma}\label{how_much_small}
The polynomials $(f,g) \in \mathcal{U}$ have degree $O(1)^{\bar{n}+1}$ and their coefficients have bitsizes $\tau 2^{O(m+n^2)}$, where $\tau$ is an upper bound on the bitsizes of the entries in $A^i$, $C$, and $b$. Furthermore, there exist a Thom encoding $\sigma$, $u \in \mathcal{U}$, and $\gamma = 2^{\tau 2^{O(m+n^2)}}$ such that $(u,\sigma)$ describes the central path for all $\mu \in (0, 1/\gamma)$. 
\end{lemma}
\begin{proof}
The output of~\cite[Algorithm~12.18]{BPR06} is a set of $(\bar{n}+3)$-tuples of polynomials $(f, g)$ in $\mathbb{Z}[\mu,T]_{\le O(\deg_V(\tilde{Q}))^{\bar{n}+1}}$, where $\deg_V(\tilde{Q}) = 4$ is the degree of $\tilde{Q}$ with respect to $V$. The bound on the bitsizes of the coefficients follows from Lemma~\ref{central_path_ball} and~\cite[Algorithm~12.18]{BPR06}.

\vspace{5px}
\noindent
Since a central solution is an isolated solution%
\footnote{This follows from the implicit function theorem.} of $\zero\big(Q(\mu),\mathbb{R}^{\bar{n}}\big)$, the projections of the real points associated to $u$ to the first $\bar{n}$ coordinates contain the central path when $\mu \in (0,1]$. Since there are finitely many $(\bar{n}+3)$-tuples of polynomials in $\mathcal{U}$, there must exist $(f,g) \in \mathcal{U}$ and $\mu_0 > 0$ such that $(f,g)$ describes the central solutions for all $\mu \in (0,\mu_0)$.

\vspace{5px}
\noindent
Let $((f,g),\sigma)$ be a real univariate representation for which ~\eqref{central_path_selection} is true (i.e., it describes the central path for sufficiently small $\mu$), where $(f,g) \in \mathbb{Z}[\mu,T]^{\bar{n}+2}$ (after discarding $g_{\bar{n}+1}$), and consider the following formulas:
\begin{align*}
\Phi_x (\mu) = (\exists T) (\forall \Lambda) \ \big(\sign(f^{(j)})=\sigma(f^{(j)}), \ j \in \mathbb{Z}_{\ge 0}\big) 
\wedge \big(\neg (C_x(T,\Lambda) = 0) \vee (\Lambda > 0)\big),\\
\Phi_s (\mu) = (\exists T) (\forall \Lambda) \  \big(\sign(f^{(j)})=\sigma(f^{(j)}), \ j \in \mathbb{Z}_{\ge 0}\big) 
\wedge \big(\neg (C_s(T,\Lambda) = 0) \vee (\Lambda > 0)\big).
\end{align*}
\noindent
Notice that $\RR(\Phi_x,\mathbb{R}) \cap \RR(\Phi_s,\mathbb{R}) \neq \emptyset$. By Theorem~\ref{14:the:tqe}, $\Phi_x$ and $\Phi_s$ are equivalent to quantifier-free $\mathcal{P}_x$- and $\mathcal{P}_s$-formulas, where $\mathcal{P}_x,\mathcal{P}_s \subset \mathbb{Z}[\mu]_{\le 2^{O(m+n^2)}}$. By~\cite[Lemma~10.3]{BPR06} and Lemma~\ref{central_path_ball}, there exists $\gamma = 2^{\tau' +O(m+n^2)}$ with $\tau'=\tau 2^{O(m+n^2)}$ such that all nonzero real roots of polynomials in $\mathcal{P}_x,\mathcal{P}_s$ are bigger than $1/\gamma$  . This completes the proof. 
\end{proof}

\hide{
The cylindrical decomposition of $\mathbb{R}^2$ adapted to $\mathcal{R}:=\{f,f^{(1)},\ldots,f^{\deg(f)}\}$ yields a cylindrical decomposition of $\mathbb{R}^2$ on which the cells are $\mathcal{R}$-invariant. Thus, a lower bound on the horizontal length of a cell on which $f, f^{(1)},\ldots,f^{\deg(f)}$ have constant signs leads to the desired bound on $\mu_0$. By~\cite[Proposition~8.45]{BPR06}, the elimination of $T$~\cite[Algorithm~11.1]{BPR06} gives polynomials in $\mathbb{Z}[\mu]$ of degree $O(1)^{\bar{n}+1}$, whose distinct roots, by~\cite[Corollary~10.22]{BPR06}, have a minimum distance of $1/\gamma$ with $\gamma = 2^{\tau' O(1)^{\bar{n}+1}}$, where $\tau'=O(\tau 2^{\bar{n}})$ by Lemma~\ref{central_path_ball}.
}

Now, we can prove the complexity of Algorithm~\ref{alg:sampling}.

%%%%%%%%
%New Theorem
%%%%%%%%
\begin{theorem}\label{uni_var_degree_of_polynomials}
There exist $\gamma = 2^{\tau 2^{O(m+n^2)}}$ and an algorithm with complexity $2^{O(m+n^2)}$ to compute $((f,g),\sigma)$ which represents the central path for all $\mu \in (0,  1/\gamma)$.
\end{theorem}
\begin{proof}
The complexity of Step~\ref{param_sampling_Step} in Algorithm~\ref{alg:sampling} is $2^{O(m+n^2)}$, which is determined by~\cite[Algorithm~12.18]{BPR06} and noting that $\tdeg_{\mu}(\tilde{Q})=6$, where $\tdeg_{\mu}(\tilde{Q})$ is the total degree of monomials in $\tilde{Q}$ containing $\mu$.  The complexity of Step~\ref{Thom_encod_Step} is determined by the number of parametrized univariate representations $(f,g)$ in $\mathcal{U}$, which is $O(\deg(\tilde{Q}))^{\bar{n}+1}$ and the complexity of~\cite[Algorithm~10.14]{BPR06} applied to every $(f,g) \in \mathcal{U}$. Step~\ref{quantifier_elim} decides the truth or falsity of~\eqref{central_path_selection} for every real univariate representation $((f,g),\sigma)$. By Theorem~\ref{thm:general_decision}, all this can be done using $2^{O(m+n^2)}$ arithmetic operations. The second part of the theorem is the direct application of Lemma~\ref{how_much_small}.  
\end{proof}
%%%%%%%%
%New Section
%%%%%%%%
\subsection{Puiseux expansion of the central path}\label{Puiseux_Expansion}
Existence of the reparametrization in Problem~\ref{CP_Analytic_Extension} can be proved by invoking the output of Algorithm~\ref{alg:sampling} and Proposition~\ref{Puiseux_Newton_Thm}. By substituting the roots $t \in \mathbb{C} \langle \langle \mu \rangle \rangle$ of $f(\mu,T) = 0$, it is easy to see that the $i^{\mathrm{th}}$ coordinate of the central path can be represented by $\sum_{j = 0}^{\infty} c_{ij} \mu^{j/q_{i}}$ with $c_{ij} \in \mathbb{R}$, and $q_{i} \in \mathbb{Z}$, see e.g.,~\cite[Theorem~1.2 of Chapter IV]{W78}%
\footnote{In case of the central path, $j$ must be nonnegative, since otherwise the root would be unbounded.}. Then one can choose $\rho$ to be (an integer multiple of) the least common multiple of $q_{i}$ over $i \in \{1,\ldots,\bar{n}\}$. However, this approach does not lead to an algorithm for an effective computation of $\rho$. 

An alternative approach, which we will follow in Section~\ref{Newton_Polygon_Alg}, can be given in terms of a semi-algebraic description of the coordinates. More precisely, let $((f,g),\sigma)$ with $(f,g)\in \mathcal{U}$ being the real univariate representation of the central path for all sufficiently small $\mu$, see e.g., Theorem~\ref{uni_var_degree_of_polynomials}. Then the graph of $v_i(\mu)$, when $\mu$ is sufficiently small, can be described by the following quantified formula
\begin{align}\label{central_solution_formula}
(\exists T)  \ \big(V_i g_0-g_i = 0\big)  \wedge  \big(\sign(f^{(j)}(\mu,T))=\sigma(f^{(j)}),\ j=0,1,\ldots \big).
\end{align}
\noindent
By Theorem~\ref{14:the:tqe},~\eqref{central_solution_formula} is equivalent to a quantifier-free $\mathcal{P}_i$-formula with $\mathcal{P}_i \subset \mathbb{Z}[\mu,V_i]$. Furthermore, there exists a polynomial $P_i \in \mathcal{P}_i $ such that $P_i(\mu,v_i(\mu))=0$ for sufficiently small $\mu$ and $P_i(0,v_i^{**})=0$ (because~
\eqref{central_solution_formula} describes the graph of a semi-algebraic function). Note that $P_i \in \mathbb{Z}[\mu,V_i]$ in Theorem~\ref{14:the:tqe}, as the output of the quantifier elimination, is the product a finite number of polynomials in $\mathbb{Z}[\mu,V_i]$, see e.g., the proof of Lemma~2.5.2 in~\cite[Page 36] {BCR98}. Therefore, $P_i$ need not be irreducible over $\mathbb{C}$ (or even $\mathbb{R}$). Nevertheless, $P_i$ has an absolutely irreducible factor with real coefficients, whose zero set contains the graph of the $i^{\mathrm{th}}$ coordinate of the central path when $\mu$ is sufficiently small.
%%%%%%%%%
%New Proposition
%%%%%%%%%
\begin{proposition}\label{irreducible_polynomial}
The polynomial $P_i$ has an absolutely irreducible factor $R_i \in \mathbb{R}[\mu,V_i]$ such that $R_i(\mu,v_i(\mu))=0$.
\end{proposition}

\begin{proof}
Let $R_i \in \mathbb{C}[\mu,V_i]$ be an absolutely irreducible factor of $P_i$ whose zero set contains the graph of the $i^{\mathrm{th}}$ coordinate of the central path. Then $R_i$ can be written as
\begin{align*}
R_i = \real(R_i) + \sqrt{-1}\imag(R_i),
\end{align*}
where  $\real(R_i),\imag(R_i) \in \mathbb{R}[\mu,V_i]$ are the real and imaginary parts of $R_i$ obtained from the real and imaginary parts of their coefficients. Then $R_i(\mu,v_i(\mu)) = 0$ yields 
\begin{align*}
\real(R_i)(\mu,v_i(\mu))=\imag(R_i)(\mu,v_i(\mu)) = 0,
\end{align*}
which in turn implies that $\real(R_i)$ and $\imag(R_i)$ have a common factor, see e.g.,~\cite[Lemma of Section 1.1]{S13}. However, this would contradict the absolute irreducibility of $R_i$. 
\end{proof}
%%%%%%%%
%New Remark
%%%%%%%%
\begin{remark}
For complexity purposes, we do not include any factorization step in Algorithm~\ref{analytic_param}. Instead, we choose $P_i$ to contain only the tail end of the central path, i.e., for sufficiently small positive $\mu$. 
\end{remark}
\noindent
Let $P_i:=\sum_{j=0}^{d_i} p_{ij}(\mu) V_i^{j}$ from Theorem~\ref{14:the:tqe} applied to~\eqref{central_solution_formula}, where $p_{ij}(\mu) \in \mathbb{Z}[\mu]$ and $d_i:=\deg_{V_i}(P_i)$. By Proposition~\ref{Puiseux_Newton_Thm}, $P_i$ can be factorized as
\begin{align}\label{central_path_factorization}
P_i(\mu,V_i) = p_{i d_i}(\mu) \prod_{\ell=1}^{d_i} \big(V_i - \psi_{i \ell} (\mu)\big), \qquad i=1,\ldots,\bar{n},
\end{align}
where the ramification index of $\psi_{i \ell}$ is denoted by $q_{i\ell}$. Therefore, there exists a unique (multiple) factor $\ell_i$ in~\eqref{central_path_factorization} such that 
\begin{align}\label{central_path_Puiseux_exp}
v_i(\mu)=\psi_{i \ell_i} (\mu):=\sum_{j = 0}^{\infty} c_{ij} \mu^{j/q_{i\ell_i}}, \ \ \text{for all sufficiently small} \ \mu > 0,
\end{align}
where $c_{ij} \in \mathbb{R}$ and $c_{i0}=v_i^{**}$. 
\begin{remark}
Although $\zero(\tilde{Q}(\mu),\mathbb{R}^{\bar{n}+1})$ is bounded over all $\mu$, not every root of $P_i = 0$ near $\mu = 0$ is necessarily bounded, unless $p_{id_i}(0) \neq 0$. 
\end{remark}
\begin{remark}
Notice that $\psi_{i \ell_i} \in \mathbb{R} \langle \mu \rangle$ because $v_i(\mu)$ is semi-algebraic.
\end{remark}
Note that $P_i = 0$ may have more than one branch over $\mu = 0$ with the same center $(0,v_i^{**})$, because $P_i$ is not necessarily irreducible over $\mathbb{C}\{\mu\}$%
\footnote{Remember that the zero set of an irreducible factor of $P_i$ over $\mathbb{C}\{\mu\}$ is a branch of $P_i = 0$.}, see Example~\ref{joint_center}. However, one of these branches contains the graph of the $i^{\mathrm{th}}$ coordinate of the central path, when $\mu$ is sufficiently small, see Proposition~\ref{Puiseux_Newton_Thm}.
\begin{remark}
By Proposition~\ref{Puiseux_Newton_Thm}, if $\psi_{i \ell_i}$ is not a multiple root of $P_i = 0$, then \textit{exactly} one of the branches of $P_i = 0$ contains the graph of the $i^{\mathrm{th}}$ coordinate of the central path, when $\mu$ is sufficiently small.
\end{remark}

This branch, in analogy with Example~\ref{derivatives_diverge}, is described by a set of $q_{i\ell_i}$ distinct Puiseux expansions, including~\eqref{central_path_Puiseux_exp},
\begin{align*}
 \psi_{i k} (\mu)= \sum_{j = 0}^{\infty} c_{ij} \Big(\exp\!\big(2\pi \sqrt{-1} (k-1)/q_{i\ell_i}\big)\mu^{1/q_{i\ell_i}}\Big)^j, \qquad k=1,\ldots,q_{i\ell_i}, 
\end{align*} 
which they all converge to $(0,v_i^{**})$. Thus, letting $q_i$ be the ramification index of $\psi_{i \ell_i} (\mu)$, $q$ the least common multiple of all $q_{i \ell_i}$ over $i \in \{1,\ldots,\bar{n}\}$, and $\rho$ be a positive integer multiple of $q$, then we get the series
\begin{align*}
\psi_{i \ell_i} (\mu^{\rho}) \in \mathbb{C}\{\mu\}, \quad i=1,\ldots,\bar{n},
\end{align*}
which are all convergent in a neighborhood of $\mu = 0$. This indicates that $v(\mu^{\rho})$ is analytic at $\mu = 0$, providing an affirmative answer to Problem~\ref{CP_Analytic_Extension}.

\hide{
%%%%%%%%
%New Definition
%%%%%%%%
\begin{definition}\label{ramification_index_centra_path}
We call $q$, the least common multiple of $q_{i\ell_i}$s from~\eqref{central_path_Puiseux_exp}, the ramification index of the central path.
\end{definition}}

%%%%%%%%
%New Remark
%%%%%%%%
\begin{remark}
We should note that $q$ is well-defined and independent of the $P_i$ and the semi-algebraic description~\eqref{central_solution_formula}.
\end{remark}

%%%%%%%%
%New Remark
%%%%%%%%
\begin{remark}
The authors in~\cite[Page~4]{HLMZ21} indicate the analyticity of $v(\mu^{\rho})$ near $\mu = 0$ in terms of the cycle number of the central path, see also~\cite[Remark~2]{HLMZ21}. However, in contrast to~\cite{HLMZ21}, our algorithmic derivation of $\rho$ is explicit in terms of the degrees and Puiseux expansions of the defining polynomials $P_i$. 
\end{remark}

Now, we can provide the proof of Theorem~\ref{thm_reparametrization}.

\begin{proof}[Proof of Theorem~\ref{thm_reparametrization}]
By Theorem~\ref{14:the:tqe}, the quantifier elimination applied to~\eqref{central_solution_formula} returns quantifier-free formulas involving polynomials $P_i \in \mathbb{Z}[\mu,V_i]_{\le 2^{O(m+n^2)}}$. With no loss of generality, see Section~\ref{Newton_Polygon_Alg}, we can assume that 
\begin{align*}
P_i(0,V_i) \neq 0, \quad P_i(0,0)=0, \quad i=1,\ldots,\bar{n}.
\end{align*}
By Proposition~\ref{WPT}, $P_i$ is the product of a Weierstrass polynomial $W_i \in \mathbb{C} \{ \mu \}[V_i]$ of degree $\deg_{V_i}(P_i)$ and a unit $U_i \in \mathbb{C} \{ \mu \}[V_i]$. We may also assume that $W_i$ is irreducible over $\mathbb{C}\{\mu\}$ (by only considering the unique component of $W_i$ whose zero set contains the graph of the $i^{\mathrm{th}}$ coordinate of the central path, see Remark~\ref{branch_alt_def}). Now, the application of Proposition~\ref{geometric_Puiseux} to $W_i$ implies that 
\begin{align*}
q_{i \ell_i} \le \deg_{V_i}(P_i), \quad i=1,\ldots,\bar{n}. 
\end{align*}
Finally, we get the result by noting that $\deg_{V_i}(P_i)=2^{O(m+n^2)}$ and $q \le \prod_{i=1}^{\bar{n}} q_{i}$.
\end{proof}

\begin{remark}
 In the presence of the strict complementarity condition, $q=1$ must be exactly 1, i.e., $v_i^{**}$ is a simple root of $P_i(0,V_i) = 0$, since otherwise the analyticity of the central path would fail at $\mu= 0$~\cite[Theorem~1]{H02}. 
 \end{remark}
%%%%%%%%
%New Remark
%%%%%%%%
\begin{remark}
 We should note that $q_{i\ell_i}$ from two different coordinates need not be identical. For instance, the ramification indices of the coordinates of the central path in Example~\ref{lower_bound_analyticity} with $n=4$ are $1$, $2$, or $4$:
\begin{align*}
X(\mu)&=\begin{pmatrix} O(\mu) & 0 & O(\mu^{\frac 34}) & O(\mu^{\frac 12})\\0 & O(\mu^{\frac 34}) & 0 & 0\\O(\mu^{\frac 34}) & 0 & O(\mu^{\frac 12}) & O(\mu^{\frac 14})\\O(\mu^{\frac 12}) & 0 & O(\mu^{\frac 14}) & 1 \end{pmatrix}, \ \ 
y(\mu)=\begin{pmatrix} 0 \\ O(\mu^{\frac 14}) \\ O(\mu^{\frac 12}) \\ O(\mu)\end{pmatrix},\\
S(\mu)&=\begin{pmatrix} 1 & 0 & O(\mu^{\frac 14}) & O(\mu^{\frac 12})\\0 & O(\mu^{\frac 14}) & 0 & 0\\O(\mu^{\frac 14}) & 0 & O(\mu^{\frac 12}) & 0\\O(\mu^{\frac 12}) & 0 & 0 & O(\mu) \end{pmatrix}. 
\end{align*}
\end{remark}
%%%%%%%%
%New Section
%%%%%%%%
\section{A symbolic algorithm based on the Newton-Puiseux theorem}\label{Newton_Polygon_Alg}
In this section, we address the complexity of computing a feasible $\rho$ using Algorithm~\ref{analytic_param}, which applies the Newton-Puiseux theorem to $P_i$. 

Given the real univariate representation $((f,g),\sigma)$ from Algorithm~\ref{alg:sampling} (after discarding $g_{\bar{n}+1}$), we choose the semi-algebraic description~\eqref{central_solution_formula} and apply the quantifier elimination algorithm to obtain a finite set $\mathcal{P}_i \subset \mathbb{Z}[\mu,V_i]$. We then identify a polynomial $P_i\in \mathcal{P}_i$ such that $P_i(\mu,v_i(\mu))=0$ for sufficiently small $\mu$. This can be done by computing $\SI(\Der(f),\mathcal{V}_{ij})$, see Section~\ref{sign_cond_Thom_encod}, where
\begin{align}\label{branch_identification_by_central_path}
 \mathcal{V}_{ij}:=\zero\Big(g^{\deg_{V_i}(R_{ij})}_0R_{ij}\big(\mu,g_i/g_0\big),\mathbb{R} \langle \mu \rangle\Big)   
\end{align}
 for every $R_{ij} \in \mathcal{P}_i$ and then checking $\sigma \in \SI(\Der(f),\mathcal{V}_{ij})$. 
 
 \subsection{Newton-Puiseux algorithm}\label{Newton_Puiseux_algorithm}
By the the proof of the Newton-Puiseux theorem, the roots of $P_i = 0$ near $\mu = 0$ are constructed as 
\begin{align*}
\varphi_{is}(\mu)=a_{i1} \mu^{\gamma_{i1}} + a_{i2} \mu^{\gamma_{i1} + \gamma_{i2}}+a_{i3} \mu^{\gamma_{i1} + \gamma_{i2}+ \gamma_{i3}}+\cdots,
\end{align*}
where $\varphi_{is}$ corresponds to the segment $s$ (described by the equation $y+\gamma_{i1} x = \beta_{i1}$) of the Newton polygon of $P_i$, $\gamma_{i1} \in \mathbb{Q}$ is the negative of the slope of the segment $s$ of the Newton polygon of $P_i$, and $a_{i1} \in \mathbb{C}$ is a (multiple) root of a polynomial in $\mathbb{Z}[T]$ by which the terms of the lowest order of $\mu$ in the polynomial 
\begin{align}\label{Newton_Puiseux_Steps_1}
P_i(\mu,\mu^{\gamma_{i1}}(V_i+T))
\end{align}
vanish (see~\cite[(3.4) on Page~98]{W78}). Notice that every segment $s$ of the Newton polygon of $P_i$ determines the order $\gamma_{i1}$ of a set of Puiseux series, as root(s) of $P_i = 0$ near $\mu = 0$. This procedure continues by forming the Newton polygon for
\begin{align}\label{Newton_Puiseux_Steps_j}
P_i^{(j+1)}\!:=\mu^{-\beta_{ij}} P_i^{(j)}(\mu,\mu^{\gamma_{ij}}(V_i+a_{ij})), \quad j=1,2,\ldots,
\end{align}
where $P_i^{(1)}:=P_i$, choosing a segment $\gamma_{i(j+1)} > 0$ (which always exists), and finding a (multiple) root $a_{i(j+1)}$ of a polynomial by which the terms of the lowest order of $\mu$ in the polynomial 
\begin{align}\label{Newton_Puiseux_Equation}
P^{j
+1}_i(\mu,\mu^{\gamma_{i(j+1)}}(V_i+T))
\end{align}
vanish. 

In particular, for the Puiseux series~\eqref{central_path_Puiseux_exp} we have $\gamma_{i1} \ge 0$ (because $\psi_{i \ell_i} (\mu)$ is convergent) and 
\begin{align*}
c_{i0} &=\begin{cases} a_{i1} \ \ &\text{if} \  \gamma_{i1} = 0\\
0 \ \ &\text{if} \ \gamma_{i1} > 0  \end{cases}\\
c_{ij} &=\begin{cases} a_{i(j+1)} \ \ &\text{if} \  \gamma_{i1} = 0\\
a_{ij} \ \ &\text{if} \ \gamma_{i1} > 0  \end{cases} & j&=1,2,\ldots.
\end{align*}

The multiplicity of $a_{ij}$ decreases monotonically, and it stabilizes at a constant integer after a finite number of iterations, say $N_i$. Then $q_i$ will be equal to the smallest common denominator of $\gamma_{i1},\ldots,\gamma_{iN_i}$ (see~\cite[Page~100]{W78}).

\begin{remark}\label{root_mult_treatment}
The multiplicity of $a_{ij}$ eventually stabilizes at 1 if $P_i = 0$ has no multiple root in $\mathbb{C} \langle \langle \mu \rangle \rangle$. By~\cite[Theorem~IV.3.5]{W78}, this can be guaranteed if $\content_{\mu}(P_{i}) = 1$ and
\begin{align*}
\deg_{V_i}(\gcd(P_i,\partial P_i/\partial V_i))=0,
\end{align*}
where $\content_{\mu}(P_{i})$ is the content of $P_i \in \mathbb{Z}[\mu][V_i]$, i.e., the greatest common divisor of the coefficients of $P_i$ in $\mathbb{Z}[\mu]$.
\end{remark}

\hide{
\begin{definition}\label{singular_part}
The \textit{singular part} of the Puiseux expansion $\varphi_{is}$ is a finite truncation
\begin{align*}
\sum_{j =1}^{N_i} a_{ij} \mu^{\gamma_{i1} +\cdots+ \gamma_{ij}},
\end{align*}
where $N_i$ is the least integer for which $a_{ij}$ is a simple root of~\eqref{Newton_Puiseux_Equation} for every $j \ge N_i$. 
\end{definition}
}

\subsection{Symbolic computation}
We apply a symbolic version of the Newton-Puiseux algorithm in~\cite[Algorithm~1]{W00} to compute $q_i$ for all bounded Puiseux expansions of $P_i(\mu,V_i)=0$ with limit equal to $v^{**}_i$. The idea of~\cite[Algorithm~1]{W00} is to compute the exponents $\gamma_{ij}$~\cite[Algorithm~1, Step~1]{W00} and then carry the roots symbolically using the minimal polynomials $S_{ij} \in \mathbb{Z}[T]$ of $a_{ij}$~\cite[Algorithm~1, Step~6]{W00} and the minimal polynomials $Z_{ij} \in \mathbb{Z}[T]$ of the primitive elements $\alpha_{ij}$ of the extension field $\mathbb{Q}(a_{i1},\ldots,a_{ij})$, $\alpha_{ij}$ being an algebraic integer such that $\mathbb{Q}(\alpha_{ij})=\mathbb{Q}(a_{i1},\ldots,a_{ij})$~\cite[Algorithm~1, Step~9]{W00}.

\hide{
\begin{remark}
Complexity analysis of computing a truncation of a Puiseux expansion has been performed in~\cite{Ch86} and also~\cite{KT78}, although the complexity bound in~\cite{Ch86} is not explicitly determined. In contrast to~\cite{Ch86,KT78}, the complexity analysis in~\cite{W00} is explicit and involves the size of the coefficients of $P_i$.
\end{remark}
}

\vspace{5px}
\noindent
For the purpose of computing the ramification indices, the Newton-Puiseux algorithm in~\cite[Algorithm~1]{W00} continues until $a_{ij}$ becomes a simple root of ~\eqref{Newton_Puiseux_Equation}. To that end, see Remark~\ref{root_mult_treatment}, we replace $P_i$ by $P_{i}/\content_{\mu}(P_{i})$, and we set $P_i:=\mu^{\alpha_i} P_i(\mu,V_i/\mu^{\theta_i})$, where $\alpha_i$ and $\theta_i$ are nonnegative integers satisfying
\begin{align}\label{prelim_transformation}
\begin{cases}
\alpha_i + o(p_{id_i}) = \theta_id_i, \\
\alpha_i + o(p_{ij}) \ge j \theta_i, \qquad j=1,\ldots,d_i-1, 
\end{cases}
\end{align}  
where $d_i=\deg_{V_i}(P_i)$. This technique~\cite[Page~247]{KT78} ensures the boundedness of the roots of $P_i = 0$. We also assume that $P_i$ is a square-free polynomial for every $i=1,\ldots,\bar{n}$. Notice that if
\begin{align*}
\deg_{V_i}(\gcd(P_i,\partial P_i/\partial V_i))>0,
\end{align*}
then $P_i$ has a multiple factor in $\mathbb{Z}(\mu)[V_i]$~\cite[Proposition~4.15]{BPR06}, and by~\cite[Theorem~I.9.5]{W78} $P_i$ has also a multiple factor in $\mathbb{Z}[\mu,V_i]$. In this case, we can compute the separable part of $P_i$ in $\mathbb{Z}[\mu,V_i]$ (which is also square-free) by applying~\cite[Algorithm~10.1]{BPR06} to $P_i$ and $\partial P_i/\partial V_i$, see~\cite[Corollary~10.15]{BPR06}. 

\begin{proposition}
The Newton-Puiseux algorithm in~\cite[Algorithm~1]{W00} computes the ramification indices in $2^{O(m+n^2)}$ iterations.
\end{proposition}
\begin{proof}
The maximum number of iterations follows from the bound 
\begin{align*}
N_i \le 4\deg_{\mu}(P_i) \deg_{V_i}(P_i)^2
\end{align*}
in~\cite[Page~1170]{W00}.
\end{proof}

\hide{
we then replace $P_i$ by $P_i/\gcd(P_i,\partial P_i/\partial V_i)$}

 \hide{
%%%%%%%%
%New Remark
%%%%%%%%
\begin{remark}
A polynomial complexity bound was derived in~\cite{KT78} for computing $N$ terms of the Puiseux expansion. This procedure builds on the singular part of a Puiseux expansion as follows
\begin{align*}
\varphi_{i s} (\mu)=\sum_{j = 1}^{N_i-1} a_{ij} \mu^{\gamma_{i1}+\ldots+\gamma_{ij}} + \mu^{\gamma_{i1}+\ldots+\gamma_{iN_i}} \tilde{\varphi}_{i s}(\mu^{1/q_{is}}), 
\end{align*}
in which $\tilde{\varphi}_{i s}(\mu)$ is the algebraic function corresponding to $\bar{P}_i(\mu,V_i):=\mu^{-\beta_{N_i}}P(\mu^{q_{is}},\mu^{\gamma_{iN_i}}V_i)$. The algebraic function $\tilde{\varphi}_{i s}(\mu)$ can be constructed by a \textit{Newton-like iteration}
\begin{align}\label{Newton_Like_iter}
W^{(k+1)}(\mu) \equiv W^{(k)}(\mu) - \bar{P}_i(\mu,W^{(k)}(\mu))/\bar{P}'_i(\mu,W^{(k)}(\mu)) \ \ \mod \ \ \mu^{2^{k+1}},
\end{align}
where $W^{(k+1)}$ are convergent power series. In~\eqref{Newton_Like_iter}, $k$ indicates the number of terms in the approximation of $\varphi_{i s} (\mu)$, and $W^{(k)} \to \tilde{\varphi}_{i s}$ as $k \to \infty$~\cite[Corollary~5.1]{KT78}. 
\end{remark}}

\subsection{On computing the optimal $\rho$}\label{Rho_Computation}
In order to obtain the optimal $\rho$, one would still need to identify a factor in~\eqref{central_path_factorization} which describes the graph of the $i^{\mathrm{th}}$ coordinate of the central path. However, this identification cannot be made by solely using the truncation of a Puiseux expansion (although we can generate as many terms as we want using the technique in~\cite{KT78}). More precisely, the Newton-Puiseux algorithm only computes a truncation of the Puiseux expansion of $v_i(\mu)$, which, in general, will not satisfy~\eqref{central_path_representation} exactly. On the other hand, we should also recall that $P_i = 0$ may have more than one branch over $\mu = 0$ with the same center $(0,v_i^{**})$, because $P_i$ is not necessarily irreducible over $\mathbb{C}\{\mu\}$, see Example~\ref{joint_center}. Thus, an optimal $\rho$ may not be always obtained using this approach. 

To get as smallest feasible $\rho$ as possible, we first check the irreducibility of $P_i$. If it holds, then $\zero(P_i,\mathbb{C})$ has exactly one branch over $\mu = 0$, which contains the graph of the $i^{\mathrm{th}}$ coordinate of the central path. Otherwise, we identify all branches of $P_i = 0$ which have a center at $(0,v_i^{**})$. We use the following technical result in Algorithm~\ref{analytic_param} to decide whether $P_i$ is irreducible over $\mathbb{C}\{\mu\}$.

\hide{
By Proposition~\ref{WPT}, the degree of an irreducible Weierstrass polynomial $f \in \mathbb{C}\{X\}[Y]$ determines the ramification index of the Puiseux expansions of $f = 0$ near $x = 0$. In particular, this result can be applied to $P_{i} \in \mathbb{Z}[\mu,V_i]$.} 

\begin{proposition}\label{irreducibility_test}
Suppose that all zeros of $P_i$ are bounded and $P_i$ is square-free. Then $P_i$ is irreducible over $\mathbb{C}\{\mu\}$ iff $P_i = 0$ has a Puiseux expansion with a ramification index equal to $\deg_{V_i}(P_i)$.
\end{proposition}
\begin{proof}
\hide{
In order to apply Proposition~\ref{WPT}, we need to assume that $\mathbf{0}$ is the limit point of the central path (i.e., by making the translation $V_i-v^{**}_i$). To that end, we define the polynomial
\begin{align*}
\bar{P}_i:= P_{i}(\mu,V_i+v_i^{**}) \in \mathbb{R}[\mu][V_i].
\end{align*}

By Proposition~\ref{WPT}, $\bar{P}_i$ can be written uniquely as the product of a unit element in $\mathbb{C}\{\mu\}[V_i]$ and a Weierstrass polynomial $W_i \in \mathbb{C}\{\mu\}[V_i]$ where $\deg(W_i)=o(\bar{P}_{i}(0,V_i))$. Notice that if $\bar{P}_i$ is irreducible over $\mathbb{C}\{\mu\}$, then  $W_i$ must be irreducible as well. 
}
By the assumptions and Proposition~\ref{Puiseux_Newton_Thm}, a Puiseux expansion with ramification index $q_{is}$ implies a branch $(T^{q_{is}},\phi(T))$ and $q_{is}$ distinct roots of $P_i = 0$. Thus $q_{is} = \deg_{V_i}(P_i)$ implies that $P_i = 0$ has exactly one branch over $\mu = 0$, since otherwise $P_i = 0$ would have more than $\deg_{V_i}(P_i)$ bounded roots. Analogously, $q_{is} < \deg_{V_i}(P_i)$, implies that $P_i=0$ must have more than one branch and thus $P_i$ must have more than one factor in $\mathbb{C}\{\mu\}[V_i]$.
\end{proof}
\noindent
By Proposition~\ref{irreducibility_test}, if $P_i$ is irreducible over $\mathbb{C}\{\mu\}$, then $q_i=\deg_{V_i}(P_i)$. Otherwise, we compute the product of ramification indices over all bounded Puiseux expansions of $P_i = 0$ with limit $v_i^{**}$, i.e., only segments with slopes $\le -\theta_i$, see~\eqref{prelim_transformation}, which yield a Puiseux expansion with $S_{i1}(v_i^{**}) = 0$. Let $(\bar{u},\bar{\sigma})$ with $\bar{u}:=\!\big(\bar{f},\big(\bar{g}_0,\bar{g}_{1},\ldots,\bar{g}_{\bar{n}}\big)\big)\! \in \mathbb{Z}[T]^{\bar{n}+2}$ be the real univariate representation of $v^{**}$ from~\cite[Algorithm~3.2]{BM22}, and let \begin{align}\label{branch_identificatioin_by_limit_point}
\mathcal{V}_i^{**}:=\zero\Big(\bar{g}^{\deg(S_{i1})}_0S_{i1}\big(\bar{g}_i/\bar{g}_0\big), \mathbb{R}\Big).
\end{align}
Then given a truncation of a Puiseux expansion of $P_i =0$ from~\cite[Algorithm~1]{W00}, the truth or falsity of $S_{i1}(v_i^{**}) = 0$ can be decided by computing $\SI(\Der(\bar{f}),\{0\})$ if $\gamma_{i1} > \theta_i$ or  $\SI(\Der(\bar{f}),\mathcal{V}_i^{**})$ if $\gamma_{i1} =\theta_i$ and then checking the inclusion $\bar{\sigma} \in \SI(\Der(\bar{f}),\{0\})$ or $\bar{\sigma} \in \SI(\Der(\bar{f}),\mathcal{V}_i^{**})$.

\vspace{5px}
\noindent
Finally, we compute $\rho$ as the least common multiple of all $\rho_i$, where $\rho_i$ is the product of all distinct $q_{is}$ corresponding to the above Puiseux expansions of $P_i = 0$. The outline of the above procedure is summarized in Algorithm~\ref{analytic_param}.

\begin{remark}
It is clear that $\rho=\prod_{i=1}^{\bar{n}} (\deg_{V_i}(P_i)!)$ will be a feasible integer for Problem~\ref{CP_Analytic_Extension}. However, analogous to the proof of Theorem~\ref{thm_reparametrization}, our goal is to compute the smallest possible feasible $\rho$.  
\end{remark}

\begin{remark}
It is worth noting that Algorithm~\ref{analytic_param} outputs the optimal $\rho$ as long as the branch of $P_i = 0$ containing the graph of the $i^{\mathrm{th}}$ coordinate of the central path is isolated for every $i=1,\ldots,\bar{n}$. In particular, an optimal $\rho$ is obtained if $P_i$ is irreducible over $\mathbb{C}\{\mu\}$ for all $i=1,\ldots,\bar{n}$.
\end{remark}

%%%%%%%%%
%New Algorithm
%%%%%%%%%
\begin{algorithm}
\caption{Reparametrization based on the Newton-Puiseux theorem}
\label{analytic_param}
\begin{algorithmic}
\State Input: The polynomial $\tilde{Q} \in \mathbb{Z} [\mu,V_1,\ldots,V_{\bar{n}+1}]$; An empty set $\mathcal{L}_i$ for every $i=1,\ldots,\bar{n}$ \\
\Return A feasible $\rho$ for which $\big(X({\mu^{\rho}),y(\mu^{\rho}),S(\mu^{\rho}})\big)$ is analytic at $\mu = 0$

\vspace{5px}
\State \textbf{Procedure:}
\begin{enumerate}
\item\label{Rep_Step} Apply Algorithm~\ref{alg:sampling} to $\tilde{Q}$ and let the output be $((f,g),\sigma)$.
\item Apply~\cite[Algorithm~3.2]{BM22} to $\tilde{Q}$ and let the output be $((\bar{f},\bar{g}),\bar{\sigma})$.
\item\label{Quant_Elim_Step} For every $i=1,\ldots,\bar{n}$, apply the quantifier elimination algorithm~\cite[Algorithm 14.5]{BPR06} with input~\eqref{central_solution_formula}. Let the output be a finite set $\mathcal{P}_i$ of polynomials in $\mathbb{Z}[\mu,V_i]$.
\item\label{Polynomial_Identification_Step} For every $i=1,\ldots,\bar{n}$ and every $R_{ij} \in \mathcal{P}_i$ apply~\cite[Algorithm~10.13]{BPR06} (Univariate Sign Determination) with input $g^{\deg_{V_i}(R_{ij})}_0R_{ij}\big(\mu,g_i/g_0\big)$ and $\Der(f)$, and let $\SI(\Der(f),\mathcal{V}_{ij})$ be the output, see~\eqref{branch_identification_by_central_path}. If $\sigma \in \SI(\Der(f),\mathcal{V}_{ij})$, then choose $R_{ij}$ as the polynomial whose zero set contains the graph of the $i^{\mathrm{th}}$ coordinate of the central path for sufficiently small $\mu$. Set $P_i := R_{ij}$.
\item\label{Boundedness_Step} For every $i=1,\ldots,\bar{n}$ apply the procedure~\eqref{prelim_transformation} to $P_i$.
\item\label{Content_Step} For every $i=1,\ldots,\bar{n}$ compute $\content_{\mu}(P_{i})$ by applying~\cite[Algorithm~10.1]{BPR06} iteratively to the coefficients of $P_i$ (as a polynomial in $V_i$). Then replace $P_{i}$ by $P_{i}/\content_{\mu}(P_{i})$.
\item\label{Sq_Free_Step} For every $i=1,\ldots,\bar{n}$ apply~\cite[Algorithm~10.1]{BPR06} with inputs $P_i,\partial P_i/\partial V_i \in \mathbb{Z}[\mu][V_i]$ and replace $P_i$ by its separable part.
\item\label{irreducible_Step} For every $i=1,\ldots,\bar{n}$ apply the symbolic Newton-Puiseux algorithm~\cite[Algorithm~1]{W00} to only one segment of the Newton polygon of $P_{i}$. If for the given segment the ramification index of the Puiseux expansion is equal to $\deg_{V_i}(P_i)$, then set $\rho_i=\deg_{V_i}(P_i)$. Otherwise, go to Step~\ref{Newton_Polygon_Step}.  
\item\label{Newton_Polygon_Step} For every $i$ failing the condition of Step~\ref{irreducible_Step}, apply the symbolic Newton-Puiseux algorithm~\cite[Algorithm~1]{W00} to $P_{i}$ for the rest of segments $s$ with negative slope of magnitude $\ge \theta_i$, see~\eqref{prelim_transformation}. Let the output be $q_{is}$, $\gamma_{i1}$, and the minimal polynomial $S_{i1} \in \mathbb{Z}[T]$.
\item\label{Convergence_Step} Given $\gamma_{i1}$ and $S_{i1}$ for every $i$ from Step~\ref{Newton_Polygon_Step}, if $\gamma_{i1} > \theta_i$, then apply~\cite[Algorithm~10.11]{BPR06} (Sign Determination) with input $\{0\}$ and $\Der(\bar{f})$. If $\gamma_{i1} = \theta_i$, then apply~\cite[Algorithm~10.13]{BPR06} (Univariate Sign Determination) with input $\bar{g}^{\deg(S_{i1})}_0S_{i1}\big(\bar{g}_i/\bar{g}_0\big)$ and $\Der(\bar{f})$. Let the output be $\SI(\Der(\bar{f}),\{0\})$ or $\SI(\Der(\bar{f}),\mathcal{V}_i^{**})$, respectively, see~\eqref{branch_identificatioin_by_limit_point}. If $\bar{\sigma} \in \SI(\Der(\bar{f}),\{0\})$ or $\bar{\sigma} \in \SI(\Der(\bar{f}),\mathcal{V}_i^{**})$, then add $s$ to $\mathcal{L}_{i}$.
\item\label{Final_Comp_Step} For every $i$ from Step~\ref{Convergence_Step} compute $\rho_i:=\prod_{s \in \mathcal{L}_{i}} q_{is}$. Then compute the least common multiple of $\rho_{i}$ over $i=1,\ldots,\bar{n}$.
\end{enumerate}
\end{algorithmic}
\end{algorithm}

\begin{example}
It is easy to see from Figure~\ref{3elliptope_alg_curve} that $F = 0$ is not irreducible over $\mathbb{C}\{\mu\}$, because it involves two isolated branches over $\mu = 0$. However, Algorithm~\ref{analytic_param} still returns the optimal $\rho = 2$ for Example~\ref{derivatives_diverge}. 
\end{example}

\subsection{Complexity and the proof of correctness} The correctness of Algorithm~\ref{analytic_param} follows from the correctness of Algorithm~\ref{alg:sampling},~\cite[Algorithm~3.2]{BM22}, Theorem~\ref{14:the:tqe}, and~\cite[Algorithm~1]{W00}. Step~\ref{Polynomial_Identification_Step} identifies $P_i$ for every $i=1,\ldots,\bar{n}$. Step~\ref{Boundedness_Step} guarantees that the roots of $P_i = 0$ are all bounded, and Steps~\ref{Content_Step} and~\ref{Sq_Free_Step} guarantee that the Puiseux expansions of $P_i$ are all distinct. Step~\ref{irreducible_Step} checks the irreducibility of $P_i$ using Proposition~\ref{irreducibility_test}. Step~\ref{Newton_Polygon_Step} applies the symbolic Newton-Puiseux algorithm to every segment of the Newton polygon of $P_i$, if $P_i$ is not irreducible over $\mathbb{C}\{\mu\}$. Step~\ref{Convergence_Step} decides the truth or falsity of $S_{i1}(v_i^{**})= 0$ for all truncated Puiseux expansions obtained from Step~\ref{Newton_Polygon_Step} and thus identifies the branches centered at $(0,v_i^{**})$. 

\vspace{5px}
\noindent
Now, we can prove Theorem~\ref{Algorithm_Complexity}. First, we need the following results.
\begin{lemma}[Theorem~3.9 in~\cite{BM22}]\label{real_univariate_limit_point}
There exists an algorithm to compute the real univariate representation $(\bar{u},\bar{\sigma})$ using $2^{O(m+n^2)}$ arithmetic operations, where
\begin{align*}
\deg(\bar{f}),\deg(\bar{g})=2^{O(m+n^2)}. 
\end{align*}
\end{lemma}

\begin{lemma}\label{Puiseux_Complexity}
The ramification indices of Puiseux expansions of $P_i = 0$ can be computed using $2^{O(m+n^2)}$ arithmetic operations. 
\end{lemma}

\begin{proof}
The result is immediate from~\cite[Theorem~1]{W00} and Theorem~\ref{14:the:tqe}. The total complexity of computing a ramification index~\cite[Theorem~1]{W00} is 
\begin{align*}
(\deg_{V_i} (P_i) \cdot \deg_{\mu}(P_i)) ^{O(1)},
\end{align*}
where 
\hide{$\tau'$ is an upper bound on the bitsizes of the coefficients of $P_i$.}by Theorem~\ref{14:the:tqe} we have
\begin{align*}
\deg_{V_i} (P_i),\deg_{\mu} (P_i)=2^{O(m+n^2)}. 
\end{align*}
\end{proof}

\begin{proof}[Proof of Theorem~\ref{Algorithm_Complexity}]
The overall complexity of Algorithm~\ref{analytic_param} is dominated by the complexity of Algorithm~\ref{alg:sampling}, ~\cite[Algorithm~3.2]{BM22}, and the complexity of the quantifier elimination~\cite[Algorithm 14.5]{BPR06}. By Lemmas~\ref{how_much_small} and~\ref{real_univariate_limit_point}, Theorem~\ref{14:the:tqe}, and Theorem~\ref{uni_var_degree_of_polynomials}, Steps~\ref{Rep_Step} to~\ref{Quant_Elim_Step} run with complexity $2^{O(m+n^2)}$, and the quantifier elimination applied to~\eqref{central_solution_formula} outputs quantifier-free formulas with polynomials of degree $2^{O(m+n^2)}$ and coefficients of bitsizes $\tau 2^{O(m+n^2)}$. The complexity of Steps~\ref{Polynomial_Identification_Step} to~\ref{Sq_Free_Step} and Steps~\ref{Convergence_Step} and~\ref{Final_Comp_Step} depend on $\bar{n}$, $\card(\mathcal{P}_i)$, $\deg(R_{ij})$, $\deg(f)$, $\deg(\bar{f})$, $\card(\Der(\bar{f}))$, and $\card(\Der(f))$, and they are all $2^{O(m+n^2)}$. Steps~\ref{irreducible_Step} and~\ref{Newton_Polygon_Step} apply the Newton-Puiseux algorithm to segments of the Newton polygon of $P_i$ which, by Lemma~\ref{Puiseux_Complexity}, have the total complexity $2^{O(m+n^2)}$. The doubly exponential bound on the magnitude of $\rho$ follows from the proof of Theorem~\ref{thm_reparametrization}, $\deg_{V_i}(P_i)=2^{O(m+n^2)}$, and the inequality 
\begin{align*}
\prod_{s \in \mathcal{L}_{i}} q_{is} \le \bigg\lceil e^{\tfrac{\deg_{V_i}(P_i)}{e}} \bigg\rceil.
\end{align*}
\end{proof}

%%%%%%%%
%New Remark
%%%%%%%%
\begin{remark}
Suppose that for every $i=1,\ldots,\bar{n}$, a factorization $P_i=\prod_j W_{ij}$ is available in Step~\ref{irreducible_Step} of Algorithm~\ref{analytic_param}, where $W_{ij} \in \mathbb{C}\{\mu\}[V_i]$ is an irreducible factor of $P_i$. Then the zero set of $W_{ij} \in \mathbb{C}\{\mu\}[V_i]$, for every $j$, represents a single branch of $P_i$ over $\mu = 0$. Let us represent $W_{ij}$ as $\real(W_{ij}) + \imag(W_{ij}) \sqrt{-1}$, where $\real(W_{ij}) \in \mathbb{R}\{\mu\}[V_i] $ and $\imag(P_i) \in \mathbb{R}\{\mu\}[V_i]$ are the real and imaginary parts of $W_{ij}$, respectively, i.e., their coefficients are convergent power series obtained by real and imaginary parts of the coefficients of $W_{ij}$. Since $\mathbb{R}\{\mu\} \subset \mathbb{R}\langle \mu \rangle$ is an ordered integral domain, the Univariate Sign Determination~\cite[Algorithm~10.13]{BPR06} is again applicable, with complexity $2^{O(m+n^2)}$ in the ordered ring $\mathbb{R}\{\mu\}$, to check the inclusion of the graph of the $i^{\mathrm{th}}$ coordinate of the central path in $\zero(W_{ij},\mathbb{C}\{\mu\})$. With this modification, we would still get the optimal $\rho$ from Algorithm~\ref{analytic_param}. 

\vspace{5px}
\noindent
We are not aware of any factorization procedure over $\mathbb{C}\{\mu\}$. In practice, however, a factorization procedure over $\mathbb{C}\{\mu\}$ would only generate a finite number of terms from each factor. Thus, even with the existence of a factorization method, we would not be able to identify the branch of the central path, because only a truncation of the convergent power series, i.e., the coefficients of an irreducible factor, would be available to us.  
\end{remark}

\hide{
%%%%%%%%
%New Remark
%%%%%%%%
\begin{remark}
Obviously $o(\bar{P}_{i}(0,V_i)) = 1$ for all $i=1,\ldots,\bar{n}$ imply the analyticity of the central path at $\mu  = 0$. Hence, the above sequential procedure yields a subroutine to check the strict complementarity condition.
\end{remark}
}

%%%%%%%%
%New Section
%%%%%%%%
\section{Concluding remarks and future research}\label{conclusion}
In this paper, we studied the analyticity of the central path of SDO in the absence of the strict complementarity condition. In essence, the superlinear convergence of primal-dual IPMs rests on the analyticity of the central path at the limit point, which is guaranteed only under the stronger condition of strict complementarity. By means of the semi-algebraic description~\eqref{central_solution_formula} and the Puiseux expansions of the roots of $P_i = 0$ for every $i=1,\ldots,\bar{n}$, we showed that there exists a reparametrization $\mu \mapsto \mu^{\rho}$ such that $\big(X(\mu^{\rho}),y(\mu^{\rho}),S(\mu^{\rho})\big)$ is analytic at $\mu = 0$, in which $\rho$ attains its optimal value at $q$, i.e., the least common multiple of ramification indices of the Puiseux expansions of $v_i(\mu)$. Our semi-algebraic approach provides an upper bound $2^{O(m^2+n^2m+n^4)}$ on $q$ and leads to Algorithm~\ref{analytic_param} for an efficient computation of a feasible $\rho$. Algorithm~\ref{analytic_param} computes a feasible $\rho$, using $2^{O(m+n^2)}$ arithmetic operations, as the least common multiple of $\prod_s q_{is}$, where the product is over all distinct ramification indices corresponding to bounded Puiseux expansions with limit $v^{**}_i$. We proved that a feasible $\rho$ from Algorithm~\ref{analytic_param} is bounded by $2^{2^{O(m+n^2)}}$. In case that the polynomials $P_i$ are all irreducible over $\mathbb{C}\{\mu\}$, Algorithm~\ref{analytic_param} outputs the optimal $\rho$.  

\vspace{5px}
\noindent
\paragraph{\textit{Real analyticity of a semi-algebraic function}}
Broadly speaking, Algorithm~\ref{analytic_param} can be modified to guarantee the analyticity of any semi-algebraic function. In contrast to complex analyticity, which can be verified using the Cauchy-Riemann conditions, checking the real analyticity is a harder problem. Suppose that the graph of a bounded semi-algebraic function $f:\mathbb{R} \to \mathbb{R}$ is described by a quantified formula $\Psi$. Then, as a sufficient condition, the analyticity of $f$ at a given $x_0 \in \mathbb{R}$ is confirmed if an analog of Algorithm~\ref{analytic_param} with input $\Psi$ returns $\rho = 1$. However, the output would be inconclusive if $\rho >1$. The reason lies in the fact that neither Algorithm~\ref{analytic_param} nor its analog for an arbitrary semi-algebraic function $f$ distinguishes between different branches, and therefore $\rho$ may not be optimal. This is the subject of our future research. 

%\vspace{5px}
%\noindent
%Inspired by our symbolic procedure, one can construct a compact $\bar{n}$-complex manifold $\mathcal{M}:=\mathcal{M}_1 \times \ldots \times \mathcal{M}_{\bar{n}}$ associated with the central path of $(\mathrm{P})-(\mathrm{D})$, where $\mathcal{M}_i$ is the compact Riemann surface associated with the complex algebraic curve $P_i = 0$ for $i = 1,\ldots,\bar{n}$. The number of connected components of $\mathcal{M}_i$ is equal to the number of irreducible factors of $P_i$ over $\mathbb{C}$. Thus, there exists a unique component corresponding to the $i^{\mathrm{th}}$-coordinate of the central path whose \textit{genus} $g_i$, due to Riemann-Hurwitz formula~\cite[Theorem~4.16.3]{JS87}, is equal to
%\begin{align*}
%g_i = 1-d_i + \frac12 \sum_a \sum_{\ell} (q^{a}_{i\ell}+1),
%\end{align*}
%where the summation is over all the finite branch points of zero set of the irreducible factor, and $q^{a}_{i\ell}$ is the ramification index of the branch $\ell$ at the branch point $a$. Hence, a higher ramification index of the central path enforces a slower convergence rate and a higher Betti number $b_1(\mathcal{M})$. It would be also interesting to see how $g_i$ is related to the number of IPM iterations. 
 
%%%%%%%%
%New Section
%%%%%%%%
\section*{Acknowledgments} 
\noindent
The authors are supported by NSF grants CCF-1910441 and CCF-2128702.

\bibliographystyle{abbrv}
\bibliography{mybibfile}
\end{document}